\documentclass[12pt]{amsart}
\usepackage{a4wide}
\usepackage{amsmath}
\usepackage[utf8]{inputenc}
\usepackage{amssymb}
\usepackage{amsopn}
\usepackage{amsthm}
\usepackage{epsfig}
\usepackage{amsfonts}
\usepackage{latexsym}
\usepackage{graphicx}
\usepackage{enumerate}
\usepackage[dvipsnames]{xcolor}
\usepackage{ mathrsfs }

\newtheorem{theorem}{Theorem}[section]
\newtheorem{lemma}[theorem]{Lemma}
\newtheorem{proposition}[theorem]{Proposition}
\newtheorem{corollary}[theorem]{Corollary}
\newtheorem{main}{Theorem} 

\theoremstyle{definition}
\newtheorem{definition}[theorem]{Definition}

\theoremstyle{remark}
\newtheorem{remark}[theorem]{Remark}

\numberwithin{equation}{section}

\newcommand{\N}{\ensuremath{\mathbb{N}}}

\newcommand{\B}{\ensuremath{\mathcal{B}}}

\renewcommand{\a}{ {\mathbf{a}}}
\renewcommand{\d}{ {\mathbf{d}}}

\renewcommand{\t}{ {\mathbf{t}}}
\newcommand{\s}{ {\mathbf{s}}}

\newcommand{\E}{\mathcal{E}}

\newcommand{\set}[1]{\left\{#1\right\}}

\newcommand{\ga}{\gamma}

\newcommand{\f}{\infty}
\newcommand{\de}{\delta}

\newcommand{\al}{\alpha}
\newcommand{\lle}{\preccurlyeq}
\newcommand{\lge}{\succcurlyeq}
\renewcommand{\a}{ \mathbf{a}}
\newcommand{\si}{\sigma}

\newcommand{\ra}{\rightarrow}

\newcommand{\K}{\mathcal K}
 
\begin{document}

\title{The $\beta$-transformation with a hole at 0}
\author{Charlene Kalle}
\address{Charlene Kalle: Mathematical Institute, University of Leiden, PO Box 9512, 2300 RA Leiden, The Netherlands}
\email{kallecccj@math.leidenuniv.nl}

\author{Derong Kong}
\address{Derong Kong: Mathematical Institute, University of Leiden, PO Box 9512, 2300 RA Leiden, The Netherlands}
\email{d.kong@math.leidenuniv.nl}

\author{Niels Langeveld}
\address{Niels Langeveld: Mathematical Institute, University of Leiden, PO Box 9512, 2300 RA Leiden, The Netherlands}
\email{n.d.s.langeveld@math.leidenuniv.nl]}

\author{Wenxia Li}
\address{Wenxia Li: Department of Mathematics, Shanghai Key Laboratory of PMMP, East China Normal University, Shanghai 200062,
People's Republic of China}
\email{wxli@math.ecnu.edu.cn}
\date{\today}
\dedicatory{}

\begin{abstract}
For $\beta\in(1,2]$ the $\beta$-transformation $T_\beta: [0,1) \to [0,1)$ is defined by $T_\beta ( x) = \beta x \pmod 1$. For $t\in[0, 1)$ let $K_\beta(t)$ be the survivor set of $T_\beta$ with hole $(0,t)$ given by 
\[
K_\beta(t):=\set{x\in[0, 1): T_\beta^n(x)\not \in (0, t) \textrm{ for all }n\ge 0}. 
\]  
 In this paper we  characterise the   bifurcation set $E_\beta$ of all parameters   $t\in[0,1)$ for which  the set valued   function  $t\mapsto K_\beta(t)$ is not locally constant. We show that $E_\beta$ is a Lebesgue null set of full Hausdorff dimension for all $\beta\in(1,2)$. We prove that for Lebesgue almost every $\beta\in(1,2)$ the bifurcation set $E_\beta$ contains both infinitely many isolated and accumulation points arbitrarily close to zero. On the other hand, we show that the set of $\beta\in(1,2)$ for which $E_\beta$ contains  no isolated points has zero Hausdorff dimension. These results contrast with the situation for $E_2$, the bifurcation set of the doubling map. Finally, we give for each $\beta \in (1,2)$ a lower and upper bound for the value $\tau_\beta$, such that the Hausdorff dimension of $K_\beta(t)$ is positive if and only if $t< \tau_\beta$. We show that $\tau_\beta \le 1-\frac1{\beta}$ for all $\beta \in (1,2)$.
 \end{abstract}
\keywords{$\beta$-transformation; Hausdorff dimension; Devil's staircase; critical points; Lyndon words; Farey words.}
\subjclass[2010]{11K55, 11A63, 68R15,  26A30, 28D05, 37B10, 37E05, 37E15}
\maketitle

\section{Introduction}\label{sec:1}
In recent years open dynamical systems, i.e., systems with a hole in the state space through which mass can leak away at every iteration, have received a lot of attention. Typically one wonders about the rate at which mass leaves the system and about the size and structure of the set of points that remain, called the {\em survivor set}. In \cite{Urb86,Urb87} Urba\'nski considered $C^2$-expanding, orientation preserving circle maps with a hole of the form $(0,t)$. He studied the way in which the topological entropy of such a map restricted to the survivor set changes with $t$. To be more precise, let $g$ be a $C^2$-expanding and orientation preserving map on the circle $\mathbb R/\mathbb Z\sim [0, 1)$.   For $t\in[0, 1)$, let $K_g(t)$ be the survivor set defined by 
\[
K_g(t):=\set{x\in[0, 1): g^n(x)\notin(0, t)\textrm{ for all }n\ge 0}. 
\]
Urba\'nski proved that the function $t \mapsto h_{top}(g| K_g(t))$ is a Devil's staircase, where  $h_{top}$ denotes the topological entropy.

\vskip .2cm
Motivated by the work of Urba\'nski, we consider this situation for the $\beta$-transformation. Given $\beta\in(1,2]$, the $\beta$-transformation $T_\beta:[0,1) \to [0,1)$ is defined by $T_\beta( x) = \beta x\pmod 1$. When $\beta=2$, we recover the doubling map. In correspondence with \cite{Urb86}, set
\begin{equation}\label{eq:11}
K_\beta(t):= \{x\in [0,1) : T_\beta^n(x) \not \in (0, t) \text{ for all } n\geq0\}.
\end{equation}
The survivor set $K_\beta(t)$ splits naturally into two pieces, $K_\beta(t) = K_\beta^0(t) \cup K_\beta^+(t)$, where
\begin{equation}\label{q:K+0} \begin{split}
K_\beta^0(t) = & \{ x \in [0,1) \, : \, \exists n \ T^n_\beta(x)=0 \, \text{ and } \ T^k_\beta(x) \not \in (0,t) \text{ for all } 0 \le k < n\},\\
K_\beta^+(t) = & \{ x \in [0,1) \, : \, T^n_\beta(x) \ge t \, \text{ for all } n \ge 0 \}.
\end{split}\end{equation}
The set $K_\beta^+(t)$ occurs in Diophantine approximation. Indeed, consider the set 
\[
F_\beta(t):=\set{x\in[0, 1)\, \Big| \, T_\beta^n(x)\ge t \textrm{ for all but finitely many }n\in\mathbb N}
\]
of points $x \ge t$, such that 0 is badly approximable by its orbit under $T_\beta$. Then  $F_\beta(t)$ can be written as a countable union of affine copies of $K_\beta^+(t)$. Thus, $\dim_H F_\beta(t)=\dim_H K_\beta^+(t) $ for all $t\in[0,1)$. The approximation  properties of $\beta$-expansions have been studied by several authors. In \cite{LPWW14} the authors considered the Hausdorff dimension of the set of values $\beta>1$ for which the orbit of 1 approaches a given target value $x_0$ at a given speed. This work generalised that of \cite{PS08}, where $x_0=0$ and the speed is fixed. Other results on the Diophantine approximation properties of $\beta$-expansions can be found in \cite{Nil09,BW14,Cao14,GL15,LW16} among others.

\vskip .2cm
Further on  we show that the set valued map $\epsilon\mapsto K_\beta(\epsilon)$ is locally constant almost everywhere, i.e., for almost all $t\in[0, 1)$ there exists a $\de>0$ such that $K_\beta(\epsilon)=K_\beta(t)$ for all $\epsilon\in[t-\de, t+\de]$. Such a result was also obtained by Urba\'nski in \cite{Urb86} for $C^2$-expanding circle maps. This fact motivates the study of the \emph{right set valued bifurcation set} (simply called \emph{bifurcation set}) $E_\beta$ containing all parameters $t\in[0, 1)$ such that  the set valued map  $\epsilon\mapsto  K_\beta(\epsilon) $ is not locally constant on any right-sided neighbourhood of $t$,  i.e., 
\begin{equation}\label{q:ebeta}
E_\beta:=\set{t\in[0, 1): K_\beta(\epsilon)\ne K_\beta(t)\textrm{ for any }\epsilon> t}.
\end{equation}

\noindent The local structure of the sets $K_2(t)$ and $E_2$ was investigated in detail in \cite{Urb86,Nil09,CT}, yielding the following results.

\begin{theorem}[Urba\'nski \cite{Urb86} and Nilsson \cite{Nil09}]\label{th:12}\mbox{}
\begin{enumerate}
\item   The  bifurcation set  $E_2$    is a Lebesgue null  set of full Hausdorff dimension. \label{item3} 
\item The function $\eta_2: t\mapsto \dim_H K_2(t)$ is a Devil's staircase: 
\begin{itemize}
\item $\eta_2$ is decreasing and continuous on  $[0, \frac{1}{2}]$;
\item  $\eta_2'(t)=0$ for Lebesgue almost every $t\in[0, \frac{1}{2}]$; 
\item $\eta_2(0)=1$ and $\eta_2(\frac{1}{2})=0$. 
\end{itemize}\label{item2}
\item The topological closure $\overline{E_2}$ is a Cantor set. \label{item4}
\item $\eta_2(t)>0$ if and only if $t<\frac{1}{2}$. 
 
\end{enumerate}
\end{theorem}
Recently,   Carminati and Tiozzo considered  in  \cite{CT}  the local  H\"older exponent of $\eta_2$. They showed that the local H\"{o}lder exponent of $\eta_2$ at any point $t\in E_2$ is equal to   $\eta_2(t)$. Other results on the size and shape of survivor sets for the doubling map $T_2$ with different holes can be found in e.g.~\cite{Alc14, BY11,Det13,GS15,Sid14}. 

 \begin{figure}[h!]\centering
\includegraphics[height=5cm]{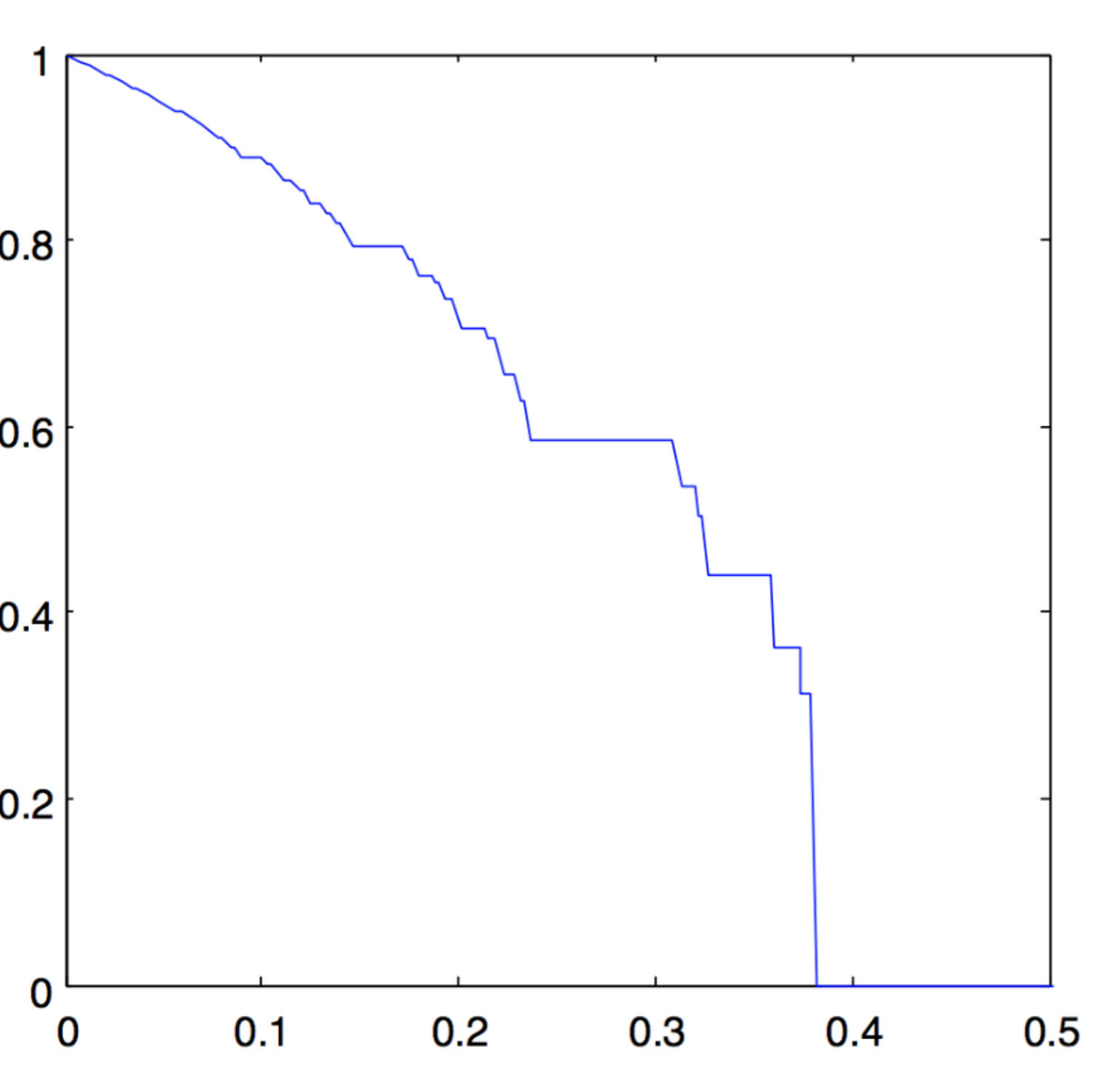}
\hspace{1.5cm}
\includegraphics[height=5cm]{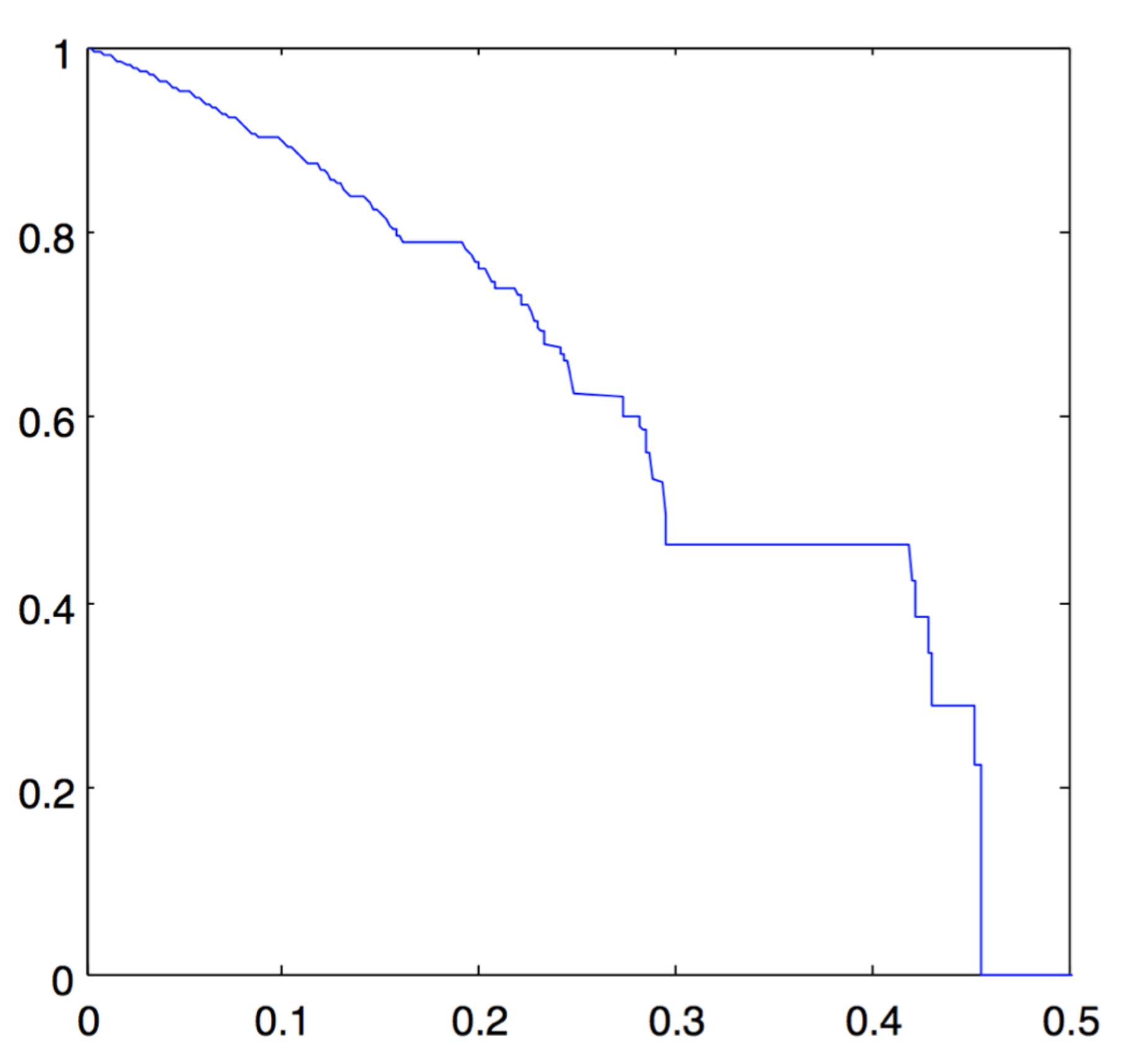}
\caption{Left: the numerical plot of $\eta_\beta$ with $\beta\approx 1.61803$ the golden ratio. Right: the numerical plot  of $\eta_\beta$ with $\beta\approx 1.83929$ the tribonacci number.}\label{fig:1}
\end{figure}

\vskip .2cm
An important ingredient for the proofs in \cite{Urb86,CT} is the fact that
\[ E_2 = \{ t \in [0,1) \, : \, T_2^n(t) \ge t \, \text{ for all } n \ge 0 \}.\]
This identity does not hold in general for $1 < \beta <2$. Therefore, we define $E_\beta^+$ by
\begin{equation}\label{q:e+}
E_\beta^+ := \{ t \in [0,1) \, : \, T_\beta^n(t) \ge t \, \text{ for all } n \ge 0 \}.
\end{equation}
In this paper we consider the survivor set $K_\beta(t)$ and the bifurcation set $E_\beta$ for $\beta\in(1, 2)$. We give a detailed description of the topological structure of $E_\beta$ and $E_\beta^+$ and their dependence on $\beta$. Theorems \ref{main:1} to \ref{main:4} below list our main results. Our first result strengthens \textit{(\ref{item2})} and \textit{(\ref{item3})} of Theorem \ref{th:12}.

\begin{main}\label{main:1}
Let $\beta\in(1, 2]$ and $t\in[0,1)$.
\begin{enumerate}
\item The bifurcation sets $E_\beta$ and $E_\beta^+$ are Lebesgue null sets of full Hausdorff dimension. 
\item The dimension function $\eta_\beta: t\mapsto \dim_H K_\beta(t)$ is a Devil's staircase:
\begin{itemize}
\item $\eta_\beta(0)=1$ and $\eta_\beta(\frac{1}{\beta})=0$;
\item $\eta_\beta$ is decreasing and continuous in $[0, \frac{1}{\beta}]$;
\item $\eta_\beta'=0$ Lebesgue almost everywhere in $[0, \frac{1}{\beta}]$.
\end{itemize}
\end{enumerate}
\end{main}
Figure \ref{fig:1} shows numerical plots of the dimension functions $\eta_\beta$ for $\beta\approx 1.61803$, the golden ratio, i.e., the real root bigger than 1 of the polynomial $x^2-x-1$ and for $\beta\approx 1.83929$, the tribonacci number, i.e., the real root bigger than 1 of the polynomial $x^3-x^2-x-1$. From the figures we can see that the value $\frac1{\beta}$ in the first item of Theorem~\ref{main:1}{\em (2)} is not sharp as a bound for the set of $x$ for which $\eta_\beta(x)=0$. Theorem~\ref{main:4} below gives more information on this bound.

\vskip .2cm The analogous statements of {\em (3)} and {\em (4)} of Theorem~\ref{th:12} for $\beta \in (1,2)$ do not always hold. The next main theorems show that in general the topological structure of $E_\beta$ differs from that of $E_2$ and that this structure depends on the value of $\beta$. Theorems~\ref{main:2} and \ref{main:3} imply that {\em (3)} of Theorem~\ref{th:12} holds only for a very small set of $\beta \in (1, 2)$.

\begin{main}\label{main:2}
For Lebesgue almost every $\beta\in(1,2)$ the bifurcation sets $E_\beta$ and $E_\beta^+$ contain infinitely many isolated and accumulation points arbitrarily close to zero and hence their closures are not Cantor sets. On the other hand,
\[
\dim_H\left(\set{\beta\in(1, 2): \exists \delta>0 \text{ such that } \overline{E_\beta^+}\cap [0,\delta] \textrm{ is a Cantor set}}\right) =1.
\] 
\end{main}

There are also infinitely many $\beta\in(1,2]$ such that $\overline{E_\beta^+}$ is a Cantor set. This is true, for example, for the countable family of multinacci numbers. In terms of Hausdorff dimension this set is small.
\begin{main}\label{main:3}
 We have $\dim_H \big( \{\beta\in(1, 2):  \overline{E_\beta^+} \textrm{ is a Cantor set}\} \big) =0$.
 \end{main}

In \cite{Cla16} Clark considered the $\beta$-transformation and characterised the holes of the form $(a,b)$ for which the survivor set $K_\beta((a, b))$ is uncountable or not. It turns out that for each $\beta \in (1,2)$, there is a unique value $\tau_\beta$, such that $\dim_H K_\beta(t)>0$ if and only if $t < \tau_\beta$. By {\em(4)} of Theorem~\ref{th:12}  we know $\tau_2 = \frac12$. We have the following result on $\tau_\beta$ for $\beta \neq 2$.
\begin{main}\label{main:4}
For each $\beta \in (1,2)$ we have $\tau_\beta \le 1-\frac1{\beta}$, and $\tau_\beta = 1-\frac1{\beta}$ if and only if $\overline{E^+_\beta}$ is a Cantor set.
 \end{main}

\vskip .2cm
In \cite{Nil} Nilsson studied the critical value $\tau_\beta'$ for the $\beta$-transformation with holes of the form $(t,1)$. In \cite[Theorem 7.11]{Nil} he proved that for each $\beta \in (1,2)$ it holds that $\tau_\beta' = 1-\frac1{\beta}$. Many of the proofs use the symbolic codings of the open systems $T_\beta$ with hole $(t,1)$. The main difficulty that we had to overcome in order to extend the results from the doubling map to the $\beta$-transformation is that the $\beta$-transformation is not coded by the full shift on two symbols. In fact, for most values of $\beta$, the associated symbolic system is not even sofic. This might also explain the difference between the result from Theorem~\ref{main:4} and the result from \cite[Theorem 7.11]{Nil}.

\vskip .2cm

The paper is arranged as follows. In Section \ref{sec:2} we introduce some notation, we recall some basic properties of $\beta$-expansions and prove Theorem~\ref{main:1}. In Section \ref{sec:3} we consider the topological structure of $E_\beta$ and $E_\beta^+$ and prove Theorem \ref{main:2}. By means of Lyndon words we construct infinitely many nested basic intervals which cover the interval $(1,2)$ up to a Lebesgue null set. We can determine all isolated points of $E_\beta^+$ by determining in which intervals it falls. The largest of these intervals are then associated to Farey words, the properties of which allow us to prove Theorem \ref{main:3} in Section \ref{sec:4} and Theorem \ref{main:4} in Section \ref{sec:5}.

\section{Preliminaries, $\beta$-expansions and first properties of $K_\beta(t)$ and $E_\beta$}\label{sec:2}
In this section we introduce some notation about sequences that we will use throughout the paper, we will recall some basic properties of $\beta$-transformations and we give some basic results on $K_\beta(t)$ and $E_\beta$. We also prove Theorem~\ref{main:1}.

\subsection{Notation on sequences}
Let $\set{0,1}^\mathbb{N}$ be the set of sequences of 0's and 1's and let $\si$ be the left shift on $\set{0,1}^\mathbb{N}$ defined by $\si((x_i))=(x_{i+1})$. We use $\{0,1\}^*$ to denote the set of all finite strings of elements from $\{0,1\}$, called \emph{words}. A word $w \in \{0,1\}^n$ is called a {\em prefix} of a sequence $(x_i) \in \{0,1\}^{\mathbb N}$ if $x_1 \ldots x_n = w$. For a word $w=w_1\ldots w_n\in\set{0, 1}^*$ we write $w^+:=w_1\ldots w_{n-1}(w_n+1) $ if $w_n=0$, and we write $w^-:= w_1w_2 \ldots w_{n-1} (w_n-1)$ if $w_n=1$. Furthermore, we use $\overline w$ to denote the \emph{reflection word} $\overline w:= (1-w_1)(1-w_2) \ldots (1-w_n)$.

\vskip .2cm
Throughout the paper we use  the lexicographical ordering  $\prec, \lle, \succ$ and $\lge$  between sequences and words, which is defined as follows. For two sequences $(x_i), (y_i)\in\set{0,1}^\mathbb{N}$ we write $(x_i)\prec (y_i)$ or $(y_i)\succ (x_i)$ if there is a smallest $m\in\N$ such that  $x_m<y_m$. Moreover, we say $(x_i)\lle (y_i)$ or $(y_i)\lge (x_i)$   if $(x_i)\prec (y_i)$ or $(x_i)=(y_i)$.   This definition can be extended to words in the following way. For $u,v \in \{0,1\}^*$, we write $u \prec v$ if and only if $u0^{\infty} \prec v0^{\infty}$.

\vskip .2cm 
Let $\# A$ denote the cardinality of the set $A$. For a subset $\mathcal Y\subseteq\set{0,1}^\N$, let $\B_n(\mathcal Y)$ denote the set of all words of length $n$ that occur in a sequence in $\mathcal Y$. The \emph{topological entropy} of $\mathcal Y$ is then given by
\[ h(\mathcal Y):=\lim_{n\ra\f}\frac{\log\#\B_n(\mathcal Y)}{n} = \inf_{n} \frac{\log\#\B_n(\mathcal Y)}{n},\]
since by the definition of $\B_n(\mathcal Y)$  the sequence $(\log \# \B_n(\mathcal Y))$ is sub-additive. Here and throughout the paper we will use the base 2 logarithm. 

\vskip .2cm

 \subsection{The $\beta$-transformation and $\beta$-expansions}
Now we recall some properties of $\beta$-transformations. Let $\beta\in(1,2]$ and let the (greedy) {\em $\beta$-transformation} $T_\beta:[0,1) \to [0,1)$ be given as in the introduction, i.e., $T_\beta(x) = \beta x \pmod 1$. It has a unique ergodic invariant measure that is equivalent to the Lebesgue measure (cf.~\cite{Ren57}). This measure is the unique measure of maximal entropy with entropy equal to $\log \beta$. For each $x \in [0,1)$ the {\em greedy $\beta$-expansion} of $x$, denoted by  $b(x,\beta) = (b_i(x,\beta))$, is the sequence obtained from $T_\beta$ by setting for each $i\ge 1$,
 \[
 b_i(x,\beta)=\left\{\begin{array}{lll}
 0,&\textrm{if}&  T^{i-1}_\beta(x)\in[0, \frac{1}{\beta}),\\
 1,&\textrm{if}&T^{i-1}_\beta (x)\in[ \frac{1}{\beta}, 1).
 \end{array}\right.
 \]   
The name greedy $\beta$-expansion stems from the fact that it is the lexicographically largest sequence $ (x_i) \in \set{0,1}^\mathbb{N}$ satisfying
\begin{equation}\label{eq:21}
x = \sum_{i \ge 1} \frac{x_i}{\beta^i}=:\pi_\beta((x_i)).
\end{equation}
We write $b(1, \beta)$ for the sequence $1b(\beta-1,\beta)$.

\vskip .2cm
The set of sequences that occur as greedy $\beta$-expansions for a given $\beta$ can be characterised using quasi-greedy $\beta$-expansions. For each $x\in(0,1]$ the quasi-greedy $\beta$-expansion of $x$ is obtained dynamically by iterating the map $\widetilde T_\beta:(0,1] \to (0,1]$ given by
\[\widetilde T_\beta (x) =\left\{\begin{array}{lll}
\beta x, &\textrm{if}&x \in \big(0, \frac1\beta \big],\\
 \beta x -1, &\textrm{if}&x \in \big( \frac1\beta, 1 \big].
\end{array}\right.
\]
The only essential difference between the maps $T_\beta$ and $\widetilde T_\beta$ is the value they take at the point $\frac{1}{\beta}$. For $x\in(0, 1]$ the \emph{quasi-greedy} $\beta$-expansion $\tilde b(x, \beta)=(\tilde b_i(x, \beta))$ is then obtained by setting $\tilde b_i(x, \beta) = 0$, if $  0 <  \widetilde T_\beta^{i-1}(x) \le \frac1{\beta}$ and $\tilde b_i(x, \beta)=1$, if $\frac{1}{\beta}<\widetilde T^{i-1}(x)\le 1$.
The quasi-greedy $\beta$-expansion of $1$ plays a crucial role in what follows. For $\beta\in(1,2]$, write
\[\al(\beta) := \tilde b(1, \beta).\]
Note that if $b(x, \beta)=b_1\ldots b_n 0^\f$ with $b_n=1$, then $\tilde b(x, \beta)=b_1\ldots b_n^-\al(\beta)$. On the other hand, if $b(x, \beta)$ does not end with $0^\f$, then $b(x, \beta)=\tilde b(x, \beta)$. The following characterisation of $\al(\beta)$ can be found in \cite[Theorem 2.3]{KomLor07}.
  \begin{lemma}
    \label{lem:21}
    Let $\mathcal Q \subset \{0,1\}^{\mathbb N}$ be the set of sequences $(a_i)\in\set{0, 1}^\N$ not ending with $0^\f$ and satisfying
\begin{equation*}
a_{n+1}a_{n+2}\ldots\lle a_1a_2\ldots\quad \textrm{for all}\quad n\ge 0.
\end{equation*}
The map $ \beta\mapsto \al(\beta)$ is a strictly increasing bijection between the interval $(1,2]$ and  $\mathcal Q$.
  \end{lemma}

For a given $\beta$, the sequence $\al(\beta)$ determines the set of all greedy $\beta$-expansions in the following way. Let $\Sigma_\beta$ be the set of all greedy $\beta$-expansions of $x\in[0, 1)$. Then (cf.~\cite{Par60})
\begin{equation}\label{eq:22}
\Sigma_\beta=\set{(x_i)\in \set{0,1}^\mathbb{N}: \si^n((x_i))\prec \al(\beta) \quad \textrm{for all }n\ge 0}.
\end{equation}
Similarly, let $\widetilde \Sigma_\beta$ be the set of all quasi-greedy $\beta$-expansions of $x\in(0, 1]$. Then  
\[
\widetilde \Sigma_\beta=\set{(x_i)\in\set{0,1}^\N: 0^\f\prec \si^n((x_i))\lle \al(\beta)\quad\textrm{for all }n\ge 0}.
\]

\noindent The following result can be found in \cite{Par60} (see also \cite{KLV16}).
\begin{lemma}\label{lem:22}
Let $\beta\in(1,2]$. Then the map $x\mapsto b(x, \beta)$ is a strictly increasing bijection from $[0, 1)$ to $\Sigma_\beta$ and it is right-continuous w.r.t.~the ordering topology on $\Sigma_\beta$.\newline
On the other hand, the map $x\mapsto \tilde b(x, \beta)$ is a strictly increasing bijection from $(0, 1]$ to $\widetilde\Sigma_\beta$ and it is left-continuous w.r.t.~the ordering topology on $\widetilde\Sigma_\beta$.
\end{lemma}

\subsection{First properties of $K_\beta(t)$ and $E_\beta$}

Let $t \in [0,1)$ be given. Recall the definitions of the survivor set $K_\beta(t) = K_\beta^0(t) \cup K_\beta^+(t)$ from \eqref{eq:11} and \eqref{q:K+0}. We define the corresponding {\em symbolic survivor sets} as the set of all greedy $\beta$-expansions of elements in the sets $\mathcal K_\beta(t)$, $\mathcal K_\beta^0(t)$ and $\mathcal K_\beta^+(t)$ respectively. Lemma~\ref{lem:22} gives the following descriptions:
\begin{equation}\label{eq:23}
\begin{split}
\mathcal K_\beta^+(t) = \,& \{ (x_i)\in \{0,1\}^{\mathbb N}  :  b(t, \beta) \lle \si^n((x_i))\prec \al(\beta) \ \forall \, n \ge 0 \},\\
\mathcal K_\beta^0(t) = \,&  \{ (x_i)\in \{0,1\}^{\mathbb N}  :  \exists \, n \ge 0 \ \, \si^n((x_i))=0^\infty \\
& \hspace{.5cm} \text{ and }  b(t, \beta) \lle \si^k((x_i))\prec \al(\beta) \ \forall \, 0 \le k <n\},\\
\mathcal K_\beta(t) = \, & \mathcal K_\beta^+(t) \cup \mathcal K_\beta^0(t).
\end{split}
\end{equation}
We will often switch from $K_\beta(t)$ to $\mathcal K_\beta(t)$ and back. Note that $K_\beta(t)$ is closed and that $T_\beta$ is continuous when restricted to $K_\beta(t)$. Under the metric $d$ on $\{0,1\}^{\mathbb N}$ given by
\[ d((x_i),(y_i)) = \beta^{-\inf\{ n \ge 1 \, : \, x_n \neq y_n\}},\]
the map $\pi_\beta: (\mathcal K_\beta(t), \sigma) \to (K_\beta(t), T_\beta)$ is a topological conjugacy. This gives that
\[h_{top} (T_\beta | K_\beta(t)) = h_{top} (\mathcal K_\beta(t)).\]

For the bifurcation set $E_\beta$, defined in \eqref{q:ebeta}, the following description can implicitly be found in \cite{Urb86}: 
\begin{proposition}\label{p:EK}
$E_\beta= \set{t\in[0, 1): t \in K_\beta(t) } $ and thus $E_\beta\cap[t, 1)\subseteq  K_\beta(t)$ for any $t \in (0,1)$.
\end{proposition}

\begin{proof}
For all $t \in (0,1)$ it holds that $t \not \in K_\beta(\epsilon)$ for any $\epsilon >t$. Hence, if $t \in K_\beta(t)$, then $t\in E_\beta$. Suppose that $t \not \in K_\beta(t)$, i.e., there is an $N \ge 1$, such that $T_\beta^N(t)\in (0, t)$. By the right-continuity of $T_\beta^N$, there is a $\de>0$ such that  
\[T_\beta^N(\epsilon)\in\Big(T_\beta^N(t),\frac{ T_\beta^N(t)+t}{2}\Big)\subseteq (0, t)\quad\textrm{for all }\epsilon\in[t, t+\de].\]
 This implies that $K_\beta(t)\cap[t, t+\de]=\emptyset$ and thus, $K_\beta(t+\de) \subseteq K_\beta(t)\subseteq K_\beta(t+\de)$. We conclude that  the function $\epsilon\mapsto K_\beta(\epsilon)$ is constant on $[t, t+\de]$.
\end{proof}

\begin{corollary}\label{c:open}
For each $\beta \in (1,2)$ the set $[0,1) \setminus E_\beta$ is open.
\end{corollary}

\begin{proof}
Let $t \not \in E_\beta$. The proof of the previous proposition then gives a $\delta_1 >0$, such that $[t, t+\delta_1] \cap E_\beta = \emptyset$. From $t \not \in K_\beta(t)$ it follows that there is an $N \ge 1$, such that $T^N_\beta (t) \in (0,t)$. Hence $T^k_\beta(t) \neq \frac1\beta$ for all $0 \le k \le N$, which means that $T^N_\beta$ is left-continuous in $t$. Then as in the proof of Proposition~\ref{p:EK} we can find a $\delta_2>0$, such that $ [t-\delta_2, t] \cap E_\beta = \emptyset$.
\end{proof}

In \eqref{q:e+} the set $E^+_\beta$ was defined. By the same proof as given for Proposition~\ref{p:EK} we also get that $E^+_\beta$ is the bifurcation set of $K_\beta^+(t)$, i.e.,
\[ E_\beta^+ = \{ t \in [0,1) \, : \, t \in K^+_\beta(t) \} = \{ t \in [0,1) \, : \, K_\beta^+ (\epsilon) \neq K_\beta^+(t) \text{ for any } \epsilon >t\} .\]
Just as for $K_\beta(t)$ we add a third set $E_\beta^0$ of the elements in $E_\beta$ that are pre-images of 0:
\begin{equation*}
E_\beta^0 = \{ t \in E_\beta \, : \, \exists \, n \ge 0 \  \, T^n_\beta (t) =0\} = \{ t \in [0,1) \, : \, t \in K^0_\beta(t) \}.
\end{equation*}
Then $E_\beta = E_\beta^+ \cup E_\beta^0$ and $E_\beta^+ \cap E_\beta^0 = \{0\}$.

\vskip .2cm
The symbolic bifurcation sets, i.e., the sets of all greedy $\beta$-expansions of elements in $ E_\beta$, $E_\beta^+$ and $E_\beta^0$ can be described as follows:
\begin{equation}\label{eq:24}
\begin{split}
\mathcal E_\beta^+ = & \, \set{(t_i) \in \{0,1\}^{\mathbb N}: \forall n \ge 0 \ \, (t_i) \lle \si^n((t_i))\prec \al(\beta)},\\
\mathcal E_\beta^0= &\, \{ (t_i) \in \{0,1\}^{\mathbb N}: \exists n \ge 0 \ \, \si^n((t_i)) =0^\infty \\
& \hspace{.5cm} \text{ and } (t_i) \lle \si^k ((t_i))\prec \al(\beta)\textrm{ for all } 0 \le k <n\},\\
\mathcal E_\beta = & \, \mathcal E_\beta^+ \cup \mathcal E_\beta^0.
\end{split}
\end{equation}

\vskip .2cm
In the series of papers \cite{Rai89,Rai92,Rai94}, Raith studied invariant sets for piecewise monotone expanding maps on the interval $[0,1]$. More specifically, in \cite{Rai94} he removed a finite number of open intervals from $[0,1]$ and considered piecewise monotone expanding maps restricted to the survivor set. He then studied the dependence on the endpoints of the holes of the Hausdorff dimension of the survivor set and of the topological entropy of the map restricted to the survivor set. Since no $x \in [0,1)$ has $T_\beta (x)=1$, we can apply these results to $T_\beta$ on $[0,1)$ with the single hole $(0,t)$ removed. In particular, applying the results from \cite[Corollary 1.1 and Theorem 2]{Rai94} give the following.
\begin{proposition}[\cite{Rai94}]\label{p:Raith}
Let $\beta \in (1,2)$ be given. The maps $H_\beta: t \mapsto h_{top} (\mathcal K_\beta(t))$ and $\eta_\beta: t \mapsto \dim_H K_\beta(t)$ are continuous on $[0,1)$.
\end{proposition}

In the process of proving \cite[Theorem 2]{Rai94}, Raith proved in \cite[Lemma 3]{Rai94} that Bowen's dimension formula also holds in this case, i.e., the Hausdorff dimension of the survivor set is the unique zero of the pressure function. In our setting this translates to the following dimension formula:
\begin{equation}\label{q:dimensionformula}
\dim_H K_\beta(t) = \frac{h_{top}(T_\beta | K_\beta(t))}{\log \beta}.
\end{equation}

Since for any $t \in [0,1)$ the sets $K_\beta^0(t)$ and $E_\beta^0$ contain at most countably many points, we have the following properties for the sets under consideration. Let $\lambda$ denote the one dimensional Lebesgue measure.

\vskip .4cm
\begin{center}
\begin{tabular}{ll}
\hline
$\dim_H K_\beta(t) = \dim_H K_\beta^+(t)$ & $\dim_H K_\beta^0(t)=0$\\
\\
$\lambda(K_\beta(t)) = \lambda(K_\beta^+(t))$& $\lambda(K_\beta^0(t))=0$\\
\\
$\dim_H E_\beta  = \dim_H E_\beta^+ $ & $\dim_H E_\beta^0 =0$\\
\\
$\lambda(E_\beta) = \lambda(E_\beta^+)$ & $\lambda(E_\beta^0)=0$\\
\\
$h_{top} (K_\beta(t)) = \max \{ h_{top} (K_\beta^+(t)),h_{top} (K_\beta^0(t))\} \qquad$ &\\
\hline
\end{tabular}
\end{center}
\vskip .2cm

\noindent This table implies that for Theorem~\ref{main:1} {\em (1)} it is enough to consider only $E_\beta$. From Proposition~\ref{p:Raith} we also get that $t \mapsto \dim_H K_\beta^+(t)$ is continuous and that
\[ h_{top}(\mathcal K_\beta(t)) = \dim_H K_\beta^+(t) \log \beta.\]

The next result specifies the relations between the sets even further.
\begin{proposition}\label{p:htopinE+}
Let $\beta \in (1,2)$ be given. If $t \in E_\beta^+$, then $h_{top} (\K_\beta(t)) = h_{top}(\K_\beta^+(t))$.
\end{proposition}

\begin{proof}
Since $\K_\beta^+(t)\subseteq \K_\beta(t)$, it suffices to prove $h_{top}(\K_\beta^+(t))\ge h_{top}(\K_\beta(t))$. For $t=0$, there is nothing to prove. Take $t\in E_\beta^+\setminus\set{0}$ and write $(t_i):=b(t,\beta)$. Then 
\[(t_i)\lle\si^n((t_i))\prec \al(\beta)\quad\textrm{for all }n\ge 0.\]
Hence $(t_i)$ does not end with $0^\f$ and by (\ref{eq:23}) we can rewrite $\K_\beta^0(t)$ as
 \begin{equation}\label{eq:jh-1}
 \K_\beta^0(t)=\{(x_i) \, : \, \exists \, n\ge 0 \ \si^n((x_i))=0^\f \text{ and }(t_i)\prec\si^k((x_i))\prec \al(\beta) \; \forall \, 0\le k<n \}.
 \end{equation}
We claim that 
 \begin{equation*}
 |\B_k(\K_\beta^0(t))|\le\sum_{j=1}^{k+1}|\B_{j-1}(\K_\beta^+(t))|.
 \end{equation*}
Take a word $a_1\ldots a_k\in \B_k(\K_\beta^0(t))$ and without loss of generality suppose it occurs as a prefix of  a sequence $(x_i)\in \K_\beta^0(t)$, i.e., $(x_i)=a_1\ldots a_k x_{k+1}x_{k+2}\ldots$. Let $j\ge 0$ be such that $x_j=1$ and the tail $x_{j+1}x_{j+2}\ldots =0^\f$. If $j=0$, then $(x_i)=0^\f$. Avoiding this trivial case we assume $j\ge 1$, and we will prove $x_1\ldots x_{j-1}0\in \B_{j}(\K_\beta^+(t))$.
 By (\ref{eq:jh-1}) it follows that 
 \begin{equation}\label{eq:hx-1}
 t_1\ldots t_{j-i}\lle x_{i+1}\ldots x_{j-1}0\prec \al_1(\beta)\ldots \al_{j-i}(\beta)\quad\textrm{for all }0\le i<j.
 \end{equation}
 Let $i^*\le j$ be the smallest index such that $x_{i^*+1}\ldots x_{j-1}0=t_1\ldots t_{j-i^*}$. If strict inequalities in (\ref{eq:hx-1}) hold for all $i<j$, then we put $i^*=j$. Note that $(t_i)\lle \si^n((t_i))\prec \al(\beta)$ for all $n\ge 0$. Then by the minimality of $i^*$ it follows that 
 \[
 x_1\ldots x_{j-1}0 t_{j-i^*+1}t_{j-i^*+2}\ldots=x_1\ldots x_{i^*}t_1t_2\ldots \in\K_\beta^+(t).
 \]
 Observe that $x_1\ldots x_{j-1}=a_1\ldots a_{j-1}$ if $j\le k$, and $x_1\ldots x_k=a_1\ldots a_k$ if $j\ge k+1$. This implies that $a_1\ldots a_{j-1}=x_1\ldots x_{j-1}\in \B_{j-1}(\K_\beta^+(t))$ if $j \le k$ or $a_1 \ldots a_k \in \mathcal B_k(\K_\beta^+(t))$ if $j \ge k+1$ and proves the claim.
 
 \vskip .2cm
 By the claim it follows that $|\B_k(\K_\beta^0(t))|\le (k+1)|\B_k(\K_\beta^+(t))|$. Using that $\K_\beta(t)=\K_\beta^0(t)\cup\K_\beta^+(t)$ we have 
 \[
 |\B_k(\K_\beta(t))|\le (k+2)|\B_k(\K_\beta^+(t))|\quad\textrm{for all }k\ge 1.
 \]
 Taking the logarithms, dividing both sides by $k$ and letting $k \to \infty$, we conclude that $h_{top}(\K_\beta(t))\le h_{top}(\K_\beta^+(t))$, which gives the result.
\end{proof}

\subsection{The size of $E_\beta$}
The results from the previous sections are enough to prove Theorem~\ref{main:1}. We start by proving the following result, which holds for all $\beta \in (1,2)$. It covers item {\em (1)} from Theorem~\ref{main:1} as well as part of Theorem~\ref{main:2}.

\begin{proposition}\label{prop:24}
For any $\beta\in(1,2)$ the bifurcation set $E_\beta$ is a Lebesgue null set. Furthermore, $\dim_H (E_\beta\cap[0, \de])=1$ for any $\de>0$. In particular, $\dim_H E_\beta =1$.
\end{proposition}

\begin{proof}
For the first part of the statement, let $\beta\in(1,2)$ and $N \in \N$. The ergodicity of $T_\beta$ with respect to its invariant measure equivalent to the Lebesgue measure $\lambda$ implies that $\lambda$-a.e.~$x \in [0,1)$ is eventually mapped into the interval $\big(0, \frac1{N}\big)$.  Hence, the survivor set $K_\beta \big(\frac1{N}\big)$ is a Lebesgue null set for each $N \in \N$. This implies that $\lambda(E_\beta)=0$, since by Proposition~\ref{p:EK}
 \[
 E_\beta \subseteq \bigcup_{N=1}^\f K_\beta \Big(\frac1{N} \Big).
 \] 
 
 \vskip .2cm
To prove the second part, take a large integer $N\ge 1$. Let $E_{\beta, N}$ be the set of  $x \in [0,1)$ with a greedy expansion $b(x, \beta)=(b_i(x, \beta))$ satisfying $b_1(x, \beta) \ldots b_N(x, \beta)=0^N$ and such that the tails $b_{N+1}(x, \beta)b_{N+2}(x, \beta)\ldots$ do not contain $N$ consecutive zeros. It immediately follows that  $E_{\beta, N} \subseteq E_\beta$. Note that $K_{\beta}^+\big( \frac1{\beta^N} \big)$ is exactly the set of   $x \in [0,1)$ for which $b(x, \beta)$ does not have more than $N$ consecutive zeros. Hence, 
\[E_{\beta,N} = \frac1{\beta^N} K_{\beta}^+\left( \frac1{\beta^N} \right)\]
and thus $\dim_H E_{\beta,N}  = \dim_H K_{\beta}^+\big( \frac1{\beta^N} \big) = \dim_H K_{\beta}\big( \frac1{\beta^N} \big)$. Moreover, for any $\delta >0$ we can find a large integer $N$, such that $E_{\beta,n} \subseteq E_\beta \cap [0, \delta]$ for all $n \ge N$. Therefore,
  \[
  \dim_H (E_\beta \cap [0, \delta])  \ge \dim_H E_{\beta,n} =\dim_H K_{\beta}\Big( \frac1{\beta^n} \Big) 
    \]
for all $n \ge N$. By continuity of the map $\eta_\beta: t\mapsto  \dim_H K_\beta(t)$, letting $n\ra\f$ gives that 
\[ \dim_H (E_\beta \cap [0, \delta]) \ge  \dim_H K_\beta(0) =\dim_H [0, 1) =1. \qedhere\]
\end{proof}

\begin{proof}[Proof of Theorem~\ref{main:1}]
Item {\em (1)} is given by Proposition~\ref{prop:24}. For item {\em (2)}, the first bullet point is immediately clear. The fact that $\eta_\beta$ decreases immediately follows from its definition and the continuity of $\eta_\beta$ is given by Proposition~\ref{p:Raith}. Finally, the set-valued map $t \mapsto K_\beta(t)$ is locally constant Lebesgue almost everywhere, since $\lambda(E_\beta)=0$.
\end{proof}

\section{Topological structure  of $E_\beta$ }\label{sec:3}
In this section we prove Theorem~\ref{main:2}. In fact, we prove a stronger result by specifying the  set of $\beta \in (1,2)$ for which there is a $\delta>0$, such that $E_\beta^+\cap[0,\delta]$ does not contain isolated points. This is the set
\begin{equation}\label{q:c3} 
C_3:=\set{\beta\in(1,2):   \textrm{   the length of consecutive zeros in } \al(\beta)\textrm{ is bounded}}.
\end{equation}
From a dynamical point of view $C_3$ is the set of $\beta\in(1, 2)$ such that the orbit $\{\widetilde T_\beta^n(1)\}_{n=0}^\f$ is bounded away from zero. Replacing $\alpha(\beta)$ in the definition of $C_3$ by $b(1, \beta)$ gives the set called $C_3$ in \cite{Sch97}. In \cite{Sch97} Schmeling proved that this set has zero Lebesgue measure and full Hausdorff dimension. Since the two versions of $C_3$ only differ by countably many points, the same holds for our set $C_3$ from (\ref{q:c3}). We prove Theorem~\ref{main:2} using Lyndon words, which we will define next.

\subsection{Lyndon words}

Recall from \eqref{eq:24} that
\[
\E_\beta^+ =\set{(t_i)\in\set{0, 1}^\N: (t_i)\lle \si^n((t_i))\prec\al(\beta)\textrm{ for all }n\ge 0}.
\]
In other words, any sequence in $\E_\beta^+$ is the lexicographically smallest sequence in $\Sigma_\beta$ under the shift map $\si$. For this reason we recall the following definition (cf.~\cite{Lot02}).
\begin{definition}\label{def:33}
 A word $\s$ is called \emph{Lyndon} if  $\s$ is aperiodic and $\si^n( \s^\f)\lge \s^\f$ for all $n\ge 0$.
\end{definition}

The following  lemma lists some useful properties of Lyndon words. The first item easily follows and is even taken as the definition of Lyndon words in \cite{CT}. We omit the proof.

\begin{lemma}\label{lem:34}\mbox{}
\begin{itemize}
\item[{\rm (i)}]  $s_1\ldots s_m$ is a   Lyndon word if and only if 
\[ s_{i+1}\ldots s_m\succ s_1\ldots s_{m-i}\quad\textrm{for all}\quad  0<i<m.\]

\item[{\rm(ii)}]  If $s_1\ldots s_m$ is a  Lyndon word, then for any $1\le n<m$ with $s_n=0$ the word $s_1\ldots s_{n}^+$ is also    Lyndon.
\end{itemize}
\end{lemma}
\begin{proof}
To prove (ii), suppose  $s_n=0$ for some $1\le n< m$. Since 1 is a Lyndon word, the statement holds for $n=1$. If $2\le n< m$, then by (i) it follows that
\[
s_{i+1}\ldots s_{n}^+\succ s_{i+1}\ldots s_n\lge s_1\ldots s_{n-i}\quad\textrm{for all}\quad 0<i<n.
\]
Therefore, again  by (i) $s_1\ldots s_{n}^+$ is a  Lyndon word as required. 
\end{proof}

By taking $i=m-1$ in Lemma \ref{lem:34} (i) it follows that  $s_1=0$ and $s_m=1$. So any  Lyndon word of length at least two starts with 0 and ends with 1. We use Lemma \ref{lem:34} to show that any isolated point in $E_\beta^+$ has a periodic greedy $\beta$-expansion. 
 
\begin{proposition}\label{prop:36}
Let $\beta\in(1,2]$. If  $t$ is an isolated point of $E_\beta^+$, then its greedy $\beta$-expansion $b(t, \beta)$ is periodic. Moreover, no element from $E_\beta^+$ is isolated in $E_\beta$.
\end{proposition} 
 
The proof of this proposition is based on the following two lemmas. Together they say that any point in $E_\beta^+$ with aperiodic $\beta$-expansion can be approximated from below by a sequences of points in $E_\beta^+$ that have a periodic orbit under $T_\beta$.
\begin{lemma}\label{lem:37}
Let $(t_i)\in\E_\beta^+$ be an aperiodic sequence. Then for each $m\ge 1$ we have
\[
(t_1\ldots t_m)^\f\prec (t_i)\quad\textrm{and}\quad (t_1\ldots t_m)^\f\in\Sigma_\beta.
\]
\end{lemma}

\begin{proof}
Let $(t_i)\in\E_\beta^+$ be an aperiodic sequence. Then by (\ref{eq:24}) we have 
\begin{equation}\label{eq:34}
(t_i)\prec \si^n((t_i))\prec \al(\beta)\quad\textrm{for all }n\ge 1.
\end{equation}
Fix $m\ge 1$. By taking  $n=m,2m,\ldots$ in  \eqref{eq:34} it follows that
\begin{align*}
(t_1\ldots t_m)^\f&=t_1\ldots t_m(t_1\ldots t_m)^\f\\
&\lle t_1\ldots t_m t_{m+1}\ldots t_{2m}(t_1\ldots t_m)^\f\\
&\lle t_1\ldots t_{2m} t_{2m+1}\ldots t_{3m} (t_1\ldots t_m)^\f \lle \cdots \lle (t_i).
\end{align*}
Since $(t_i)$ is not periodic, we conclude that $(t_1\ldots t_m)^\f\prec (t_i)$. 

\vskip .2cm
For the second statement, \eqref{eq:34} and the first part of the proposition give that 
\begin{align*}
\si^n((t_1\ldots t_m)^\f)&=t_{n+1}\ldots t_m (t_1\ldots t_m)^\f\prec t_{n+1}\ldots t_m t_{m+1}t_{m+2}\ldots \;\prec \al(\beta)
\end{align*}
for each $0\le n<m$, hence $(t_1\ldots t_m)^\f\in\Sigma_\beta$.
\end{proof}

\begin{lemma}\label{lem:38}
Let $(t_i)\in\E_\beta^+$ be an  aperiodic sequence. Then there exist infinitely many $m\in\N$ such that $ t_1\ldots t_m$ is a Lyndon word.
\end{lemma}

\begin{proof}
Let $(t_i)\in\E_\beta^+$ be a non-periodic sequence. Then
\begin{equation}
\label{eq:35}
(t_i) \prec \si^n((t_i)) \quad\textrm{for any}\quad n\ge 1.
\end{equation}
We construct a strictly increasing sequence of integers $(m_i)$ such that $t_1\ldots t_{m_i}$ is a Lyndon word for all $i\ge 1$.

\vskip .2cm
{\bf Step I}. Fix a large integer $M_1$. If $t_1\ldots t_{M_1}$ is a Lyndon word, then we are done by setting $m_1=M_1$. Otherwise, by Lemma \ref{lem:34}  (i) there exists a smallest integer $1 \le m_1\le M_1-1$ such that
\begin{equation}
\label{eq:36}
t_{m_1+1}\ldots t_{M_1}\lle t_1\ldots t_{M_1-m_1}.
\end{equation}
We claim that $t_1\ldots t_{m_1}$ is a Lyndon word.

{ If $m_1=1$ then it is clear that $t_1$ is a Lyndon word. So we assume $m_1>1$.}   Then from \eqref{eq:35}  it follows that
\begin{equation}
\label{eq:37}
t_{i+1}\ldots t_{m_1}\lge t_1\ldots t_{m_1-i}\quad\textrm{for all }0<i<m_1.
\end{equation}
By Lemma \ref{lem:34} (i) it suffices to prove that strict inequalities hold in \eqref{eq:37} for all $0<i<m_1$. Suppose on the contrary that
$t_{k+1}\ldots t_{m_1}=t_1\ldots t_{m_1-k}$ for some $0<k<m_1$. Then by \eqref{eq:36} and \eqref{eq:35} it follows that
\begin{align*}
t_{k+1}\ldots t_{M_1}&=t_{k+1}\ldots t_{m_1} t_{m_1+1}\ldots t_{M_1}\\
&\lle t_1\ldots t_{m_1-k}t_1\ldots t_{M_1-m_1}\lle t_1\ldots t_{M_1-k},
\end{align*}
contradicting the assumption that $m_1$ is the smallest integer satisfying \eqref{eq:36}. Therefore, $t_1\ldots t_{m_1}$ is a Lyndon word.

\vskip .2cm
{\bf Step II.}  Suppose for $j\ge 2$ we have found integers $M_1<M_2<\cdots<M_{j-1}$ and $m_1<m_2<\cdots <m_{j-1}$, such that $m_i\le M_i$ and $t_1\ldots t_{m_i}$ is a Lyndon word for each $1 \le i < j$. By Lemma~\ref{lem:37} we have $(t_i)\succ (t_1\ldots t_{m_{j-1}})^\f$. This implies that there exists a large integer $M_j>M_{j-1}$ such that
\begin{equation}
\label{eq:38}
t_1\ldots t_{M_j}0^\f\succ (t_1\ldots t_{m_{j-1}})^\f.
\end{equation}
If $t_1\ldots t_{M_j}$ is a Lyndon word, then we are done by setting $m_j=M_j$. Otherwise, by Lemma \ref{lem:34} (i) let $1\le m_j\le M_j-1$ be the smallest integer for which
\begin{equation}
\label{eq:39}
t_{m_j+1}\ldots t_{M_j}\lle t_1\ldots t_{M_j-m_j}.
\end{equation}
By the same argument as in Step I we have that $t_1\ldots t_{m_j}$ is a Lyndon word. So, it suffices to prove that $m_{j}>m_{j-1}$.

\begin{itemize}
\item If $m_{j}<m_{j-1}$, then by \eqref{eq:39} and using $M_j>M_{j-1}$ it follows that $t_{m_j+1}\ldots t_{M_{j-1}}\lle t_1\ldots t_{M_{j-1}-m_j}$, leading to a contradiction with the minimality of $m_{j-1}$.

\item If $m_{j}=m_{j-1}$, then by \eqref{eq:38} and \eqref{eq:39} we have
\begin{equation}
\label{eq:310}
(t_1\ldots t_{m_j})^\f\prec t_1\ldots t_{M_j}0^\f\lle t_1\ldots t_{m_j}t_1\ldots t_{M_j-m_j}0^\f.
\end{equation}
Write $M_j=s_{j} m_j+r_j$ with $s_j\ge 1$ and $1 \le r_j\le m_j$. Then by \eqref{eq:310} we conclude that
\[
t_1\ldots t_{M_j}=(t_1\ldots t_{m_j})^{s_j}t_1\ldots t_{r_j},
\]
leading to a contradiction with \eqref{eq:38}.
\end{itemize}

Hence, we have found a strictly increasing sequence $(m_j)$ such that $t_1\ldots t_{m_j}$ is a Lyndon word for each $j\ge 1$.
\end{proof}

Note that both previous lemmas do not hold for $\mathcal E_\beta$. Let $(t_i) \in \mathcal E_\beta^0$ be such that $\sigma^n ((t_i))=0^\infty$. Then for any $m>n$ we have $(t_1 \ldots t_m)^\infty \succ (t_i)$, contradicting the statement of  Lemma~\ref{lem:37}. As for the statement of Lemma~\ref{lem:38}, for all $m \ge 2n$ we have that $t_1 \ldots t_m$ is not Lyndon.

{\begin{proof}[Proof of Proposition \ref{prop:36}]
Let $t\in E_\beta^+$ be a point with aperiodic greedy $\beta$-expansion $b(t,\beta)=(t_i)$. Since $(t_i)\in\E_\beta^+$, by Lemma \ref{lem:38} there exists a sequence  $(m_j)$ such that $t_1\ldots t_{m_j}$ is Lyndon for all $j\ge 1$. Furthermore, by Lemma \ref{lem:37} we have $(t_1\ldots t_{m_j})^\f\in\Sigma_\beta$ for each $j\ge 1$. Hence, for all  $j\ge 1$ we have $(t_1\ldots t_{m_j})^\f\in\E_\beta^+$ and thus $\pi_\beta\left((t_1\ldots t_{m_j})^\infty\right)\in E_\beta^+$. Letting $j\ra\f$ we conclude that $\pi_\beta((t_1\ldots t_{m_j})^\f)\ra \pi_\beta((t_i))=t$ which implies that $t$ is not isolated in $E_\beta^+$.

\vskip .2cm
Now assume  that $t\in E_\beta^+$ has a periodic greedy $\beta$-expansion $b(t, \beta)=(t_1\ldots t_m)^\f$, where $m$ is chosen minimal. We will show that $t$ is not isolated in $E_\beta$. If $m=1$, then we have $b(t, \beta)=0^\f$, i.e., $t=0$. In this case the result trivially follows from Proposition~\ref{prop:24}. Now assume $m\ge 2$. Let $a_1\ldots a_m$ be the maximal cyclic permutation of $t_1\ldots t_m$. Then there exists a $j\in\set{0,1,\ldots, m-1}$, such that $a_1\ldots a_m=t_{j+1}\ldots t_mt_1\ldots t_j$. Note that $\si^n((t_1\ldots t_m)^\f)\prec \al(\beta)$ for all $n\ge 0$. Then 
\begin{equation}\label{eq:1}
(a_1\ldots a_m)^\f\prec \al(\beta),
\end{equation}
which implies $a_1\ldots a_m\lle \al_1(\beta)\ldots\al_m(\beta)$. We claim that 
$
a_1\ldots a_m\prec \al_1(\beta)\ldots \al_m(\beta).
$

\vskip .2cm
If $a_1\ldots a_m=\al_1(\beta)\ldots \al_m(\beta)$, then (\ref{eq:1}) together with Lemma~\ref{lem:21} gives
\[ a_1\ldots a_m\lle \al_{m+1}(\beta)\ldots \al_{2m}(\beta) \lle \al_1(\beta)\ldots \al_m(\beta)=a_1\ldots a_m.\]
So, $a_1\ldots a_{2m}=(a_1\ldots a_m)^2$. Iterating this argument with Lemma~\ref{lem:21} and (\ref{eq:1}) gives that $\al(\beta)=(a_1\ldots a_m)^\f$, leading to a contradiction with (\ref{eq:1}). This proves the claim.

\vskip .2cm
For $N\in\N$, define the sequence $\mathbf t_N:=(t_1\ldots t_m)^N t_1\ldots t_j^+0^\f$. Since  $t_j=0$, the sequence $\mathbf t_N$ is well-defined. By Lemma~\ref{lem:37} one can verify that $\si^n(\mathbf t_N)\succ \mathbf t_N$ for all $0\le n<mN+j$. Moreover, by the claim it follows that $\si^n(\mathbf t_N)\prec \al(\beta)$ for all $n\ge 0$. So, $\mathbf t_N\in\mathcal E_\beta^0$ for all $N\in\N$. Since $\pi_\beta(\mathbf t_N)\searrow t$ as $N\ra\f$, the point $t\in E_\beta^+$ is not isolated in $E_\beta$. 
\end{proof}
}

The next proposition says that no point from $E_\beta^0$ can be approximated from above by elements from $E_\beta$ and that a point $t \in E_\beta^0$ is isolated in $E_\beta$ if the orbit of 1 enters $(0,t)$.
\begin{proposition} \label{p:e0right}
Let $t \in E_\beta^0$. Then there is a $\delta >0$, such that $E_\beta \cap [t, t+\delta] = \{t\}$. Moreover, if $\beta -1 \not \in K_\beta(t)$, then $t$ is isolated in $E_\beta$.
\end{proposition}
 
\begin{proof}
If $t \in E_\beta$, then there is a smallest $n \ge 0$, such that $T^n_\beta(t) =\frac1\beta$. By the right continuity of $T_\beta$, there is a $\delta>0$, such that all $\epsilon \in (t, t+\delta]$ satisfy $T_\beta^{n+1}(\epsilon) \in (0,t) \subseteq (0,\epsilon)$. Hence, $\epsilon \not \in K_\beta(\epsilon)$ and thus, $\epsilon \not \in E_\beta$.

\vskip .2cm
The first statement implies that to prove an element from $E_\beta^0$ is isolated, it is enough to prove that it cannot be approximated from below. If again $n$ is such that $T^n_\beta(t) = \frac1\beta$, then for a small enough $\delta$, we know that for any point $\epsilon \in [t-\delta, t)$ the point $T^{n+1}_\beta(\epsilon)$ is close to 1. Let $m$ be the smallest integer such that $T^m_\beta(\beta-1) \in (0,t)$. Then there is a $0 < \delta < t- T^m_\beta(\beta-1)$, such that any $\epsilon \in [t-\delta,t)$ satisfies
\[ T_\beta^{n+1+m+1}(\epsilon) \in (0, T_\beta^m(\beta-1)) \subseteq (0, \epsilon).\]
Hence, $\epsilon \not \in E_\beta$ and $E_\beta \cap [t-\delta,t]=\{t\}$.
\end{proof}

\subsection{The construction of basic intervals}
from now on we focus on the set $E_\beta^+$. We first construct subintervals of $(1,2)$ such that $E_\beta^+$ contains isolated points whenever $\beta$ is in one of these intervals. We start with a couple of lemmas.
 
\begin{lemma}\label{lem:39}
Let $(t_i), (\alpha_i) \in \{0,1\}^{\mathbb N}$ be given. Suppose there is an $m \ge 1$, such that $\alpha_m=1$ and $\sigma^m((\alpha_i)) \lle (t_i)$. Define the sets 
\begin{eqnarray*}
\mathcal K&:=& \{ (x_i) \in \{0,1\}^{\mathbb N} \, : \,  (t_i) \lle \sigma^n((x_i)) \prec (\alpha_i) \text{ for all } n \ge 0 \},\\
\mathcal X_m&:=& \{ (x_i) \in \{0,1\}^{\mathbb N} \, : \,  (t_i) \lle \sigma^n((x_i)) \lle (\alpha_1 \ldots \alpha_m^-)^\infty \text{ for all } n \ge 0 \}.
\end{eqnarray*}
Then $\mathcal K = \mathcal X_m$.
\end{lemma}

\begin{proof}
Obviously, $\mathcal X_m \subseteq \mathcal K$. We show that $\mathcal K \setminus \mathcal X_m = \emptyset$. Suppose that this is not the case and let $(x_i) \in \mathcal K \setminus \mathcal X_m$. Then there is a $j \ge 1$, such that $x_{j+1}\ldots x_{j+m}=\al_1 \ldots \al_m$. Since $(x_i) \in \mathcal K$, the assumption that $\sigma^m((\alpha_i)) \lle (t_i)$ implies that
\[ x_{j+m+1} x_{j+m+2} \ldots \prec \al_{m+1} \al_{m+2} \ldots \lle (t_i),\]
which contradicts $(x_i) \in \mathcal K$. Hence $\mathcal K \setminus \mathcal X_m = \emptyset$.
\end{proof}

Let $\beta\in(1,2)$ and $t\in [0, 1)$. The previous lemma has the following consequence for $\mathcal K_\beta^+(t)$. If there is a smallest $m\ge 1$ such that
\[    \al_{m+1}(\beta)\al_{m+2}(\beta)\ldots\lle b(t, \beta),\]
or equivalently, $\tilde T^m_\beta (1) \le t$, then we can rewrite $\mathcal K_\beta^+(t)$ as
  \[  \mathcal K_\beta^+ (t)=\set{(x_i): b(t, \beta)\lle \si^n((x_i))\lle (\al_1(\beta)\ldots \al_{m}(\beta)^-)^\f\textrm{  for any }n\ge 0}.
  \]
 Hence, any point in the survivor set $K_\beta^+(t)$ then has the property that its entire orbit lies between $t$ and the point $\pi_\beta((\al_1(\beta)\ldots \al_{m}(\beta)^-)^\infty)$. We need two more lemmas. Recall the definition of the set $\mathcal Q$ from Lemma~\ref{lem:21} as the set of sequences that occur as $\alpha(\beta)$ for some $\beta \in (1,2]$.

\begin{lemma}\label{lem:310}
Let $(a_1\ldots a_m)^\f \in \mathcal Q$ with $m$ minimal. Then 
\[
 a_{i+1}\ldots a_{m}^+\lle a_1\ldots a_{m-i}\quad\textrm{for all }0<i<m.
\]
\end{lemma}
\begin{proof}
Let $\beta\in(1,2)$ be such that $\al(\beta)=(a_1\ldots a_m)^\f$. Then $b(1, \beta)=a_1\ldots a_m^+0^\infty$. Hence, for each $0<i<m$ we have $b(T^i_\beta(1), \beta) = a_{i+1}\ldots a_m^+0^\infty$ and $T_\beta^i(1) <1$. The result then follows from Lemma~\ref{lem:22}.
\end{proof}

Note that for any word $a_1\ldots a_m$ there is a $0\leq j <m-1$ such that $a_{j+1}\ldots a_m a_1\ldots a_j$ is the smallest among its permutations and therefore Lyndon. We call this word the \emph{Lyndon word for} $a_1\ldots a_m$.

\begin{lemma}\label{lem:311}
Let $a_1 \ldots a_m$ be a non periodic word that is the largest among its permutations and let $s_1 \ldots s_m = a_{j+1}\ldots a_ma_1 \ldots a_j$ be the Lyndon word for it. Set
\[  \mathcal Z_m:= \{  (x_i) \in\set{0, 1}^\N\, :  \, s_1\ldots s_m 0^\f \lle \sigma^n ( (x_i)) \lle (a_1 \ldots a_m)^\f \ \forall n \ge 0 \}.\]
\begin{itemize}
\item[{\rm{(i)}}] If $(x_i)\in \mathcal Z_m$ has prefix $a_{j+1}\ldots a_m$, then $(x_i)=(s_1\ldots s_m)^\f$;
\item [{\rm(ii)}] If $(x_i)\in\mathcal Z_m$ has prefix $a_1\ldots a_j$, then $(x_i)=(a_1\ldots a_m)^\f$. 
\end{itemize}
\end{lemma}
{
\begin{proof} 
Since the proofs of (i) and (ii) are similar, we only give the proof of (i). 
Let $a_{j+1}\ldots a_m x_1x_2\ldots \in\mathcal Z_m$. Then
\begin{equation}
\label{eq:315}
s_1\ldots s_m0^\f\lle \si^n(a_{j+1}\ldots a_m x_1x_2\ldots)\lle (a_1\ldots a_m)^\f\quad\textrm{for all}\quad n\ge 0.
\end{equation}
In particular,
\[
a_{j+1}\ldots a_m x_1\ldots x_j\lge s_1 \ldots s_m = a_{j+1}\ldots a_m a_1\ldots a_j,
\]
which gives
\[
x_1\ldots x_j\lge a_1\ldots a_j.
\]
On the other hand, by taking $n=m-j$ in (\ref{eq:315}) we get $x_1\ldots x_m\lle a_1\ldots a_m$. Hence
\[
x_1\ldots x_j=a_1\ldots a_j\quad\textrm{and}\quad x_{j+1}\ldots x_m\lle a_{j+1}\ldots a_m.
\]
Again, by (\ref{eq:315}) now with $n=m$, we have $x_{j+1}\ldots x_m\lge s_1\ldots s_{m-j}=a_{j+1}\ldots a_m$. Therefore,
$x_1\ldots x_m=a_1\ldots a_m$.
By iteration we conclude that
\[
a_{j+1}\ldots a_mx_1x_2\ldots =(a_{j+1}\ldots a_ma_1\ldots a_j)^\f=(s_1\ldots s_m)^\f
\]
as required. 
\end{proof}
}

We now construct   infinitely many nested intervals  $(\beta_L, \beta_R]$ such that   $E_\beta^+$ has   isolated points whenever $\beta\in(\beta_L,\beta_R]$. Figure \ref{fig:3} shows some of these intervals. We will later show that these basic intervals cover the whole interval $(1,2)$ up to a set of zero Lebesgue measure. 

\begin{figure}[h!]\centering
\includegraphics[width=\textwidth]{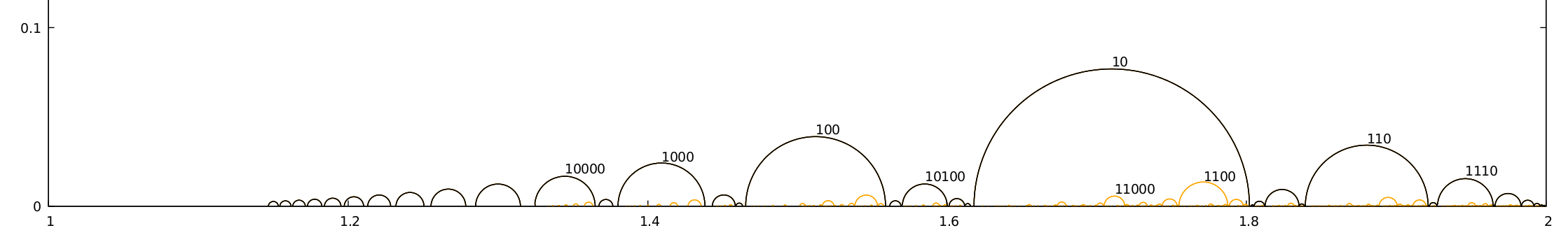}
 \caption{Some of the basic intervals  $(\beta_L, \beta_R]$. The numbers near the arches indicate the words $a_1 \ldots a_m$ such that $\alpha(\beta_L) = (a_1 \ldots a_m)^\infty$. The intervals that are not contained in any other interval are the Farey intervals. They are the ones for which $a_1\ldots a_m$ is a Farey word. The the arches corresponding to Farey intervals are shown in black, the orange arches correspond to words that are Lyndon, but not Farey.}\label{fig:3}
\end{figure}

\begin{proposition}\label{prop:312}
Let $\beta_L\in(1,2)$ be any base with periodic quasi-greedy expansion $\alpha(\beta_L)=(a_1\ldots a_m)^\infty$, where $m$ is minimal. Let $s_1\ldots s_m$ be the Lyndon word for $a_1\ldots a_m$. The point $\beta_R\in(1,2)$ with quasi-greedy expansion given by
\[
\alpha(\beta_R)=a_1\ldots a_{m}^+ (s_1\ldots s_m)^\f
\]
exists. Moreover,
\begin{itemize}
\item[{\rm(i)}] $(s_1\ldots s_m)^\infty \in \Sigma_\beta$ if and only if $\beta > \beta_L$;
\item[{\rm(ii)}] if $\beta\in  (\beta_L, \beta_R]$, then $\pi_\beta ( (s_1\ldots s_m)^\f )$ is an isolated point of $E_\beta^+$;
\item[{\rm(iii)}] if $\beta>\beta_R$, then $\pi_\beta((s_1\ldots s_m)^\f)$ is not an isolated point of  $E_\beta^+$. 
\end{itemize}
\end{proposition}

\begin{proof}
Let $\beta_L$ be as in the proposition. First we show that the interval $(\beta_L, \beta_R)$ is well-defined, i.e., $\beta_R$ exists and that $\beta_L<\beta_R$. We use the characterisation from Lemma \ref{lem:21}, so it suffices to show that  the sequence $\a=a_1\ldots a_{m}^+ (s_1\ldots s_m)^\f \in \mathcal Q$, i.e., it satisfies $\si^n(\a)\lle\a$ for all $n\ge 0$. Since $a_1 \ldots s_m$ is the Lyndon word for $a_1 \ldots a_m$, any word of length $1 \le n \le m-1$ occurring in $a_1 \ldots a_m$ is lexicographically larger than or equal to $s_1 \ldots s_n$. Combining this with Lemma~\ref{lem:310} and Lemma \ref{lem:34} (i) gives  
\[
a_{n+1}\ldots a_{m}^+ s_1\ldots s_n\lle a_1\ldots a_{m-n}a_{m-n+1}\ldots a_m\prec a_1\ldots a_{m}^+
\]for all $0<n<m$ .
So $\si^n(\a)\prec \a$ for each $0 < n < m$. Moreover, since
\begin{equation*}
\label{eq:319}
\si^n((s_1\ldots s_m)^\f) \lle   (a_1\ldots a_m)^\f\prec a_1\ldots a_{m}^+(s_1\ldots s_m)^\f
\end{equation*}
for all $n\ge 0$, we get $\si^n(\a)\prec  \a$ for all $n\ge 1$, and thus $\a \in \mathcal Q$. Lemma \ref{lem:21} then implies that $\a$ is indeed the quasi-greedy  expansion of $1$ for some base $\beta_R$, i.e., $\al(\beta_R)=a_1\ldots a_{m}^+(s_1\ldots s_m)^\f$. Since $\al(\beta_L) \prec \alpha(\beta_R)$, Lemma \ref{lem:21} also gives that $\beta_R > \beta_L$. Hence,  the interval $(\beta_L, \beta_R]$ is well-defined.
\vspace{.2cm}

Let $1 \le j \le m-1$ be such that
\[ s_1 \ldots s_m = a_{j+1} \ldots a_m a_1 \ldots a_j.\]

For (i), note that if $\beta \le \beta_L$, then $(s_1\ldots s_m)^\infty \not \in \Sigma_\beta$, since
\[ \sigma^j ((s_1\ldots s_m)^\infty) = (a_1\ldots a_m)^\infty \lge \alpha(\beta).\]
For $\beta\in(\beta_L,\beta_R]$ it follows immediately that $(s_1\ldots s_m)^\infty \in \Sigma_\beta$, since $s_1 \ldots s_m$ is the smallest permutation of $a_1 \ldots a_m$ and $(a_1\ldots a_m)^\infty \prec \alpha(\beta)$.

\vskip .2cm
For (ii), let $\beta \in (\beta_L, \beta_R]$ and set $t=\pi_\beta((s_1\ldots s_m)^\f)$.Then $b(t, \beta)=(s_1\ldots s_m)^\f\in\E_\beta^+$, so $t \in E_\beta^+$. By Lemma \ref{lem:22} and since $t$ has a periodic $\beta$-expansion, there exists a  small $\delta>0$ such that for any $x\in[t-\delta, t+\delta]$ the greedy expansion $b(x, \beta)$ has prefix $s_1\ldots s_m$. By Lemma \ref{lem:39} it follows that
\begin{equation}\label{eq:320}
\begin{split}
\K_\beta^+(t-\delta)&\subseteq \set{(x_i): s_1\ldots s_m 0^\f\lle \si^n((x_i))\prec a_1\ldots a_m^+(s_1\ldots s_m)^\f ~\forall n\ge 0}\\
&=\set{(x_i): (s_1\ldots s_m)^\f\lle \si^n((x_i))\prec a_1\ldots a_m^+(s_1\ldots s_m)^\f~ \forall n\ge 0}\\
&=\set{(x_i): (s_1\ldots s_m)^\f \lle \si^n((x_i))\lle (a_1\ldots a_m)^\f~\forall n\ge 0}\\
&=\set{(x_i): s_1\ldots s_m 0^\f\lle \si^n((x_i))\lle (a_1\ldots a_m)^\f ~\forall n\ge 0}.
\end{split}
\end{equation}
Since for any $x\in[t-\delta, t+\delta]$  the greedy expansion $b(x, \beta)$ begins with $s_1\ldots s_m$, by Lemma \ref{lem:311} (i) and (\ref{eq:320}) we obtain  that
\[
K_\beta^+(t-\delta)\cap[t-\delta, t+\delta]\subseteq\set{t}.
\]
Since $t\in E_\beta^+\cap[t-\delta, t+\delta]\subseteq K_\beta^+(t-\delta)\cap[t-\delta, t+\delta]$, we conclude that $t$ is isolated in $E_\beta^+$ for any $\beta\in(\beta_L, \beta_R]$.
\vskip .2cm

For (iii), let $\beta> \beta_R$ and again set $t=\pi_\beta((s_1\ldots s_m)^\f)$. We construct a sequence $(\t_n)$ in $\E_\beta^+$ such that $\t_n\searrow (s_1\ldots s_m)^\f$ in the order topology as $n\rightarrow \infty$.
Let 
\begin{equation}\label{eq:321}
\t_n:= ((s_1\ldots s_m)^n s_1 \ldots s_{m-j}^+)^\infty = ((a_{j+1}\ldots a_m a_1\ldots a_j)^n a_{j+1} \ldots a_m^+)^\infty.
\end{equation}

We claim that there is an $N\in \mathbb{N}$ such that $\t_n\in\E_\beta^+$ for all $n>N$. 
 Note that the largest permutation of $\t_n$ is given by
\begin{equation*}
\begin{split}
\d_n &= (a_1\ldots a_m^+ (a_{j+1}\ldots a_m a_1\ldots a_j)^{n-1} a_{j+1}\ldots a_m)^\infty\\
&= (a_1\ldots a_{m}^+ (s_1\ldots s_m)^{n-1} s_1\ldots s_{m-j})^\infty .
\end{split}
\end{equation*}
For $\beta > \beta_R$ either $\alpha_1(\beta)\ldots \alpha_m(\beta)\succ a_1\ldots a_m^+$ or there exists an $N\geq 1$, such that $\alpha(\beta)=a_1\ldots a_m^+(s_1\ldots s_m)^{N-1}b_1\ldots b_m$ with $b_1\ldots b_m\succ s_1\ldots s_m$. In the first case obviously $\d_n\prec \alpha(\beta)$. In the second case we have $\d_n\prec \alpha(\beta)$ for all $n>N$. Hence $\t_n\in \Sigma_\beta$ for all $n>N$. Moreover, from \eqref{eq:321} and Lemma \ref{lem:34} (i) it follows that $\si^i(\t_n)\lge \t_n$ for any $0<i\leq mn$ and from Lemma \ref{lem:34} (ii) it follows that the word $s_1\ldots s_{m-j}^+$ is Lyndon, which in turn implies that $\si^i(\t_n)\lge \t_n$ for any $mn <i<mn+m-j$, using Lemma \ref{lem:34} (i). Hence, $\si^i(\t_n)\lge \t_n$ for all $i\ge 0$ and thus $\t_n\in\E_\beta^+$ for all $n\ge N$. 

\vskip .2cm
We have found a sequence $(\t_n) \subseteq \mathcal E_\beta^+$ decreasing to $b(t, \beta)=(s_1\ldots s_m)^\f$ as $n\to\f$ and accordingly, a sequence $(\pi_\beta(\t_n))\subseteq  E_\beta^+$  decreasing to $t=\pi_\beta((s_1\ldots s_m)^\f)$ as $n\ra\f$. Therefore, $t$ is not isolated in $E_\beta^+$. 
\end{proof}

\subsection{Isolated points for $E_\beta$}
Recall that $C_3$ is the set of $\beta\in(1,2)$ such that the length of consecutive zeros in the   quasi-greedy expansion $\al(\beta)$ is bounded.

\begin{theorem}\label{t:notinc3}
If $\beta \in (1,2)\setminus C_3$, then both $E_\beta \cap [0,\delta]$ and $E_\beta^+\cap [0,\delta]$ contain both infinitely many isolated and accumulation points for all $\delta>0$.
\end{theorem}

\begin{proof}
By Proposition \ref{prop:24} it follows that $E_\beta \cap[0, \de]$ and $E_\beta^+\cap[0, \de]$ contain infinitely many accumulation points for all $\de>0$, so we focus on the isolated points. Fix $\beta\in(1,2)\setminus C_3$. Then $\al(\beta)$ contains consecutive zeros of arbitrary length. Hence, $\alpha(\beta)$ is not periodic and the orbit of 1 under $\widetilde T_\beta$ will come arbitrarily close to $0$. This implies that for any $t>0$, $\beta -1 \not \in K_\beta(t)$ and thus by Proposition~\ref{p:e0right} any $t \in E_\beta^0 \setminus \{0\}$ will be isolated in $E_\beta$. Note that for any $n \ge 1$ we have $\frac1{\beta^n} \in E_\beta^0$. This gives the statement for $E_\beta$.

\vskip .2cm
To prove that $E_\beta^+$ contains infinitely many isolated points arbitrarily close to 0, we construct by induction a sequence of intervals $(\beta_{L,k}, \beta_{R, k}),  k\ge 1,$ such that $\beta\in (\beta_{L,k}, \beta_{R, k})$ for all $k\ge 1$, where $(\beta_{L,k}, \beta_{R, k})$ is defined as in Proposition  \ref{prop:312}. Write
\begin{equation}\label{eq:327}
\al(\beta)=1^{l_1}0^{m_1}1^{l_2}0^{m_2}\ldots1^{l_k}0^{m_k}\ldots.
\end{equation}
Since $\al(\beta)$ does not end with $0^\f$, we have $m_k\in\set{1,2,\ldots}$ for all $k\ge 1$. Furthermore, from $\beta\notin C_3$ we get $\sup_{k\ge 1}m_k=\f$. 

\vskip .2cm
Set $i_0=1$ and let $i_1>i_0$ be the smallest index for which $m_{i_1}>m_1$. Set $\a_1:=1^{l_1}0^{m_1}\cdots 1^{l_{i_1}-1}0$. Then $\si^n(\a_1^\f)\lle\a_1^\f$ for all $n\ge 0$, and then by Lemma \ref{lem:21}  the sequence $\a_1^\f$ is the quasi-greedy expansion of $1$ for some base $\beta_{L, 1}$, i.e., $\al(\beta_{L, 1})=\a_1^\f$. Note that the word $\a_1$ contains consecutive zeros of length at most $m_1$. So the Lyndon word $\s_1=s_1\ldots s_{l_1+m_1+\cdots+ l_{i_1}}$ for $\a_1$ begins with $0^{m_1}1$. By using $m_{i_1}>m_1$ and (\ref{eq:327}) it follows that
\[
\al(\beta_{L,1})=\a_1^\f=(1^{l_1}0^{m_1}\cdots 1^{l_{i_1}-1}0)^\f\prec 1^{l_1}0^{m_1}\cdots1^{l_{i_1}}0\cdots=\al(\beta),
\]
and
\[
\al(\beta_{R, 1})=\a_1^+\, \s_1 ^\f=1^{l_1}0^{m_1}\cdots 1^{l_{i_1}}0^{m_1}1\cdots\succ1^{l_1}0^{m_1}\cdots 1^{l_{i_1}}0^{m_{i_1}}1\cdots= \al(\beta).
\]
By Lemma \ref{lem:21} we have $\beta\in (\beta_{L,1}, \beta_{R, 1})$. Moreover,  by Proposition \ref{prop:312} $\pi_\beta(\s_1^\f)$ is an isolated point of $E_\beta^+$.
Now we pick $i_k$ using $i_{k-1}$. Let $i_k>i_{k-1}$ be the smallest index such that $m_{i_k}>m_{i_{k-1}}$. Then by the definitions of $i_1, \ldots, i_{k-1}$ it follows that $m_{i_k}>m_j$ for all $j<i_k$. Set $\a_k:=1^{l_1}0^{m_1}\cdots 1^{l_{i_k}-1}0$. Then the block $\a_k$ contains consecutive zeros of length at most $m_{i_{k-1}}$. So the Lyndon word $\s_k=s_1\ldots s_{l_1+m_1+\cdots+l_{i_k}}$ for $\a_k$ begins with $0^{m_{i_{k-1}}}1$. We obtain
\begin{align*}
\al(\beta_{L, k})&=\a_k^\f  =(1^{l_1}0^{m_1}\cdots 1^{l_{i_k}-1}0)^\f\prec1^{l_1}0^{m_1}\cdots 1^{l_{i_k}}0\cdots= \al(\beta),\\
\al(\beta_{R,k})&=\a_k^+\, \s_k^\f=1^{l_1}0^{m_1}\cdots 1^{l_{i_k}}0^{m_{i_{k-1}}}1\cdots\succ 1^{l_1}0^{m_1}\cdots 1^{l_{i_k}} 0^{m_{i_k}}1\cdots=\al(\beta).
\end{align*}
Therefore, $\beta\in (\beta_{L,k}, \beta_{R,k})$ and by Proposition  \ref{prop:312} $\pi_\beta(\s_k^\f)$ is an isolated point of $E_\beta^+$.

\vskip .2cm
By induction we construct a sequence of intervals $(\beta_{L,k}, \beta_{R, k}), k\ge 1$, such that $\beta\in (\beta_{L,k}, \beta_{R, k})$ for all $k\ge 1$. Moreover, the points $\pi_\beta(\s_k^\f)$, $k\ge 1$, are isolated in $E_\beta^+$. Note that $\s_k$ begins with a block $0^{m_{i_{k-1}}}1$ for any $k\ge 1$ and $m_{i_{k-1}}$ strictly increases to $\f$ as $k\ra\f$. This implies that  $E_\beta^+\cap[0, \de]$ contains infinitely many isolated points for any $\de>0$.
\end{proof}

\begin{theorem}\label{t:inc3}
For $\beta \in C_3$ there is a $\delta >0$, such that $E_\beta^+\cap[0, \de]$ has no isolated points.
\end{theorem}

\begin{proof}
Fix $\beta\in C_3$. Then the length of consecutive zeros in $\al(\beta)$ is bounded by some large integer $M$. Set $\de=\frac1{\beta^{M+3}}=\pi_\beta(0^{M+2}10^\f)$. To show that $E_\beta^+\cap[0, \de]$ has no isolated points, suppose on the contrary that $t$ is an isolated point of $E_\beta^+\cap[0, \de]$. By Proposition \ref{prop:36} it follows that the greedy $\beta$-expansion $b(t, \beta)$ of $t$ is periodic, namely
\begin{equation*}
b(t, \beta)=(t_1\ldots t_m)^\f\in\E_\beta^+
\end{equation*}
with minimal period $m$. Moreover, $t_1\ldots t_m$ is Lyndon. For $m=1$ we get that $t=0$, which by Proposition~\ref{prop:24} is not isolated in $E_\beta^+$. Let $m\geq 2$ and let $a_1\ldots a_m$ be the maximal cyclic permutation of $t_1\ldots t_m$. Then $(a_1\ldots a_m)^\f \in \mathcal Q$, so by Lemma \ref{lem:21} it is the quasi-greedy expansion of $1$ for some base $\beta_L$, i.e., $\al(\beta_L)=(a_1\ldots a_m)^\f$. By Proposition \ref{prop:312} it follows that $\beta\in (\beta_L, \beta_R]$, where $\beta_R$ is the unique base satisfying
\[
\al(\beta_R)=a_1\ldots a_m^+(t_1\ldots t_m)^\f.
\]
Hence,
\begin{equation}\label{eq:329}
(a_1\ldots a_m)^\f\prec \al(\beta)\lle a_1\ldots a_m^+(t_1\ldots t_m)^\f.
\end{equation}
Since $t\le \de=\pi_\beta(0^{M+2}10^\f)$, we have $(t_1\ldots t_m)^\f=b(t, \beta)\lle 0^{M+2}10^\f$. So $t_1\ldots t_m$ begins with $M+2$ consecutive zeros and $a_1\ldots a_m$ contains $M+2$ consecutive zeros. Hence, by (\ref{eq:329}) we conclude that $\al(\beta)$ contains $M+1$ consecutive zeros, leading to a contradiction with our hypothesis that the number of consecutive zeros in $\al(\beta)$ is bounded by $M$.
\end{proof}

\begin{proof}[Proof of Theorem~\ref{main:2}]
The first part of the statement follows from Proposition~\ref{prop:24} and Theorem~\ref{t:notinc3}, since $\lambda (C_3)=0$ by the results from \cite{Sch97}. The fact from \cite{Sch97} that $\dim_H C_3=1$ together with Theorem~\ref{t:inc3} gives the last part of the result.
\end{proof}

\section{When $E_\beta^+$ does not have isolated points}\label{sec:4}
In this section we prove Theorem \ref{main:3}, which states that the set of $\beta\in(1,2)$ for which $E_\beta^+$ has no isolated points is rather small, it has zero Hausdorff dimension. The theorem is obtained by showing that the intervals $(\beta_L,\beta_R]$ introduced in the previous section cover all but a Hausdorff dimension zero part of the interval $(1,2)$. Figure \ref{fig:3} suggests that the basic intervals are nested. In Proposition~\ref{prop:313} below we prove that this is indeed the case. Subsequently, we identify those intervals $(\beta_L, \beta_R]$ that are not contained in any other basic interval, which turn out to be te ones given by a specific subset of the Lyndon words, called Farey words.

\begin{proposition}\label{prop:313}
Let $I_1=(\beta_L, \beta_R]$ and $I_2=(\tilde \beta_L, \tilde\beta_R]$ be two different basic intervals.  If $I_1\cap I_2\ne \emptyset$, then $I_1\subset I_2$ or $I_2\subset I_1$.
\end{proposition}

\begin{proof}
Suppose $I_1=(\beta_L, \beta_R]$ is parameterised by the  word $a_1\ldots a_m$ and $I_2=(\tilde \beta_L, \tilde\beta_R]$ is parameterised by the word $b_1\ldots b_n$, i.e., 
\begin{align*}
\al(\beta_L)&=(a_1\ldots a_m)^\f, \quad \al(\beta_R)=a_1\ldots a_m^+(s_1\ldots s_m)^\f;\\
\al(\tilde\beta_L)&=(b_1\ldots b_n)^\f, \quad \al(\tilde\beta_R)=b_1\ldots b_n^+(t_1\ldots t_n)^\f,
\end{align*}
where $s_1\ldots s_m$ is the Lyndon word for $a_1\ldots a_m$, and $t_1\ldots t_n$ is the Lyndon word for $b_1\ldots b_n$. 
Since $I_1\cap I_2\ne\emptyset$, by symmetry we may assume $\tilde\beta_L\in I_1=(\beta_L, \beta_R]$. We are going to show that $\tilde\beta_R<\beta_R$, which by Lemma \ref{lem:21} is equivalent to showing
\begin{equation}\label{eq:323}
b_1\ldots b_n^+(t_1\ldots t_n)^\f\prec a_1\ldots a_m^+(s_1\ldots s_m)^\f.
\end{equation} 
Since $\beta_L<\tilde\beta_L\le \beta_R$, by Lemma \ref{lem:21} it follows that 
\begin{equation}
\label{eq:324}
(a_1\ldots a_m)^\f\prec (b_1\ldots b_n)^\f\lle a_1\ldots a_m^+(s_1\ldots s_m)^\f.
\end{equation}
We claim that $n> m$.
\begin{itemize}
\item If $n<m$, then by (\ref{eq:324}) we have $b_1\ldots b_n=a_1\ldots a_n$. Write $m=u n+r$ with $u\ge 1$ and $1 \le r\le n$. By Lemma \ref{lem:310} and (\ref{eq:324}) it follows that $a_1 \ldots a_{un} = (b_1 \ldots b_n)^u$ and $b_1 \ldots b_r = a_1 \ldots a_r = a_{un+1} \ldots a_m^+$, so
\[
a_1\ldots a_m=(b_1\ldots b_n)^u b_1\ldots b_r^-.
\]
By using that $s_1\ldots s_m$ is the Lyndon word for $a_1\ldots a_m$ we obtain  that 
\begin{align*}
a_1\ldots a_m^+(s_1\ldots s_m)^\f&=(b_1\ldots b_n)^u b_1\ldots b_r(s_1\ldots s_m)^\f\\
&\lle (b_1\ldots b_n)^u b_1\ldots b_r(b_{r+1}\ldots b_n b_1\ldots b_r^-(b_1\ldots b_n)^{u-1}b_1\ldots b_r)^\f\\
&\prec (b_1\ldots b_n)^\f,
\end{align*}
leading to a contradiction with (\ref{eq:324}).

\item If $n=m$, then by (\ref{eq:324}) we have $b_1\ldots b_m=a_1\ldots a_m$ or $b_1\ldots b_m=a_1\ldots a_m^+$. Both cases contradict (\ref{eq:324}). 
\end{itemize}

Therefore, $n>m$.  Write $n=k m+j$ with $k\ge 1$ and $1 \le j\le m$. By (\ref{eq:324}) we have 
\[ b_1\ldots b_n\lle a_1\ldots a_m^+(s_1\ldots s_m)^{k-1} s_1\ldots s_j.\]
From 
\[
s_{j+1}\ldots s_m s_1\ldots s_j\lle a_1\ldots a_m\prec a_1\ldots a_m^+
\]
one can easily see that 
\[
(a_1\ldots a_m^+(s_1\ldots s_m)^{k-1}s_1\ldots s_j)^\infty \succ  a_1\ldots a_m^+(s_1\ldots s_m)^\f \lge (b_1 \ldots b_n)^\infty.
\]
So $b_1\ldots b_n\neq a_1\ldots a_m^+(s_1\ldots s_m)^{k-1}s_1\ldots s_j$ and hence,
\begin{equation}
\label{eq:326}
b_1\ldots b_n^+\lle a_1\ldots a_m^+(s_1\ldots s_m)^{k-1}s_1\ldots s_j.
\end{equation}
If strict inequality holds in (\ref{eq:326}), then (\ref{eq:323}) follows immediately and we are done. Suppose that the equality holds in (\ref{eq:326}).
We split the proof of (\ref{eq:323}) into the following two cases. 
\newline
\textbf{(I)} $1\le j\le \frac{m}{2}$. 
Since $s_1\ldots s_m$ is a Lyndon word, it follows that 
\[s_1\ldots s_j^-\prec s_1\ldots s_j\lle s_{j+1}\ldots s_{2j}. \]
Furthermore, $t_1\ldots t_n$ is the Lyndon word for
\[ b_1\ldots b_n=a_1\ldots a_m^+(s_1\ldots s_m)^{k-1}s_1\ldots s_j^-.\]
Then
\begin{align*}
(t_1\ldots t_n)^\f&\lle (s_1\ldots s_j^- a_1\ldots a_m^+(s_1\ldots s_m)^{k-1})^\f\\
&\prec (s_{j+1}\ldots s_{2j}s_{2j+1}\ldots s_m s_1\ldots s_j)^\f.
\end{align*}
 By (\ref{eq:326}) this proves (\ref{eq:323}) as required.
\newline
\textbf{(II)} $\frac{m}{2}<j\le m$. Since $s_1\ldots s_m$ and $t_1\ldots t_n$ are both Lyndon words, by Lemma \ref{lem:34} (i) it follows that 
 \begin{align*}
 (t_1\ldots t_n)^\f&\lle (s_1\ldots s_{m-j} s_{m-j+1}\ldots s_j^- a_1\ldots a_m^+(s_1\ldots s_m)^{k-1})^\f\\
 &\prec (s_{j+1}\ldots s_m s_1\ldots s_j)^\f.
 \end{align*}
 Again, we established (\ref{eq:323}).
\end{proof}

\subsection{Farey words}
The set of Farey words is constructed recursively as follows. Let $F_0$ be the ordered set containing the two words 0 and 1, i.e., $F_0:=(0,1)$. For each $n \ge 1$, $F_n = ( v_1, \ldots , v_{2^n+1})$ is the ordered set obtained from $F_{n-1} = (w_1, \ldots, w_{2^{n-1}+1})$ by:
\[ \begin{array}{llll}
v_{2i-1}&:=w_i \quad &&\text{for } 1\leq i\leq 2^{n-1}+1, \\
v_{2i}&:=w_iw_{i+1} \quad && \text{for } 1\leq i \leq 2^{n-1},
\end{array}\]
where $w_iw_{i+1}$ denotes the concatenation of the words $w_i$ and $w_{i+1}$. For example,
\[ F_0=(0,1), \quad  F_1 = (0,01,1), \quad F_2 = (0,001,01,011,1).\]
Then a word  $w\in \{0,1\}^*$ is a {\em Farey word} if there is an $n\geq 0$ such that $\omega\in F_n$. For each $n\ge 0$ the words in $F_n$ are listed from left to right in a lexicographically increasing order (cf.~\cite[Lemma 2.2]{CIT13}). In particular, no Farey word is periodic. Let
\[ \mathcal F:=\bigcup_{n\geq 0} F_n\setminus\set{0, 1}. \]
be the set of  \emph{non-degenerate} Farey words. Clearly, any $w_1\ldots w_m\in \mathcal F$ has $w_1=0 = 1-w_m$. We recall  from \cite[Proposition 2.3]{CIT13}  the following definition.
\begin{definition}
\label{def:41}
Let $w=w_1\ldots w_m\in \mathcal F$. A decomposition $w=uv$ is called the \emph{standard factorisation} of $w$ if $u$ and $v$ are both Farey words.
\end{definition}
By the construction of $F_n$ the standard factorisation of a non-degenerate  Farey word $w_1\ldots w_m$  is \emph{unique}.
We list some properties of Farey words. The proofs can be found in \cite[Propositions 2.8 and 2.9]{CIT13}.
\begin{enumerate}
\item[(f1)] For $w_1\ldots w_m\in\mathcal F$, both $w_1\ldots w_{m-1}0$ and $1w_2\ldots w_m$ are palindromes, i.e.,
\[
w_2\ldots w_{m-1}=w_{m-1}\ldots w_2.
\]

\item[(f2)]  Suppose $w_1\ldots w_m\in\mathcal F$ has standard factorisation $(w_1\ldots w_{m_1})(w_{m_1+1}\ldots w_m)$. The lexicographically  largest cyclic permutation of $w_1\ldots w_m$ is given by
\[
w_{m-m_1+1}\ldots w_m w_1\ldots w_{m-m_1}=w_mw_{m-1}\ldots w_2w_1.
\]

\item[(f3)]  Suppose $w_1\ldots w_m\in\mathcal F$ has standard factorisation $(w_1\ldots w_{m_1})(w_{m_1+1}\ldots w_m)$.  Then $w_1\ldots w_m$ is a Lyndon word and its lexicographically second smallest cyclic permutation is $w_{m_1+1}\ldots w_mw_1\ldots w_{m_1}$.
\end{enumerate}

\vskip .2cm
Recall that for $w_1 \ldots w_m \in \{0,1\}^*$, $\overline{w_1\ldots w_m}=(1-w_1)(1-w_2)\ldots(1-w_m)$. By Lemma \ref{lem:34} (i) it follows that if $w_1 \ldots w_m \in \mathcal F$, then $(\overline{w_1\ldots w_m})^\f\in \mathcal Q$, i.e., $ \si^n((\overline{w_1\ldots w_m})^\f)\lle (\overline{w_1\ldots w_m})^\f$ for all $ n\ge 0$. Properties (f1), (f2), (f3) imply the following.
\begin{lemma}
\label{lem:43}
Let $\overline{a_1\ldots a_m}\in\mathcal F$. Suppose 
\[\overline{a_1\ldots a_m}=(\overline{a_1\ldots a_{m_1}})(\overline{a_{m_1+1}\ldots a_m})\] is the standard factorisation of $\overline{a_1\ldots a_m}$.
 
\begin{enumerate}

\item[{\rm(i)}]  The words $a_1\ldots a_{m-1}1$ and $0a_2\ldots a_m$ are palindromes, i.e.,
\[
a_2\ldots a_{m-1}=a_{m-1}\ldots a_2.
\]
\item[{\rm(ii)}]  The Lyndon word associated to $a_1\ldots a_m$ is given by
\[
a_{m-m_1+1}\ldots a_m a_1\ldots a_{m-m_1}=a_m a_{m-1}\ldots a_1.
\]
\item[{\rm(iii)}]  $(a_1\ldots a_{m_1})^\f\in\mathcal Q$.
\end{enumerate}
\end{lemma}
\begin{proof}
  (i) and (ii) immediately follow from (f1) and (f2) respectively. For (iii), we know that $\overline{a_1\ldots a_{m_1}}$ is a Lyndon word and therefore $(\overline{a_1\ldots a_{m_1}})^\f \lle\sigma^n((\overline{a_1\ldots a_{m_1}})^\f)$ for all $n\in\mathbb{N}$. This gives $(a_1\ldots a_{m_1})^\f \lge\sigma^n((a_1\ldots a_{m_1})^\f)$ for all $n\in\mathbb{N}$.   
\end{proof}

For Farey words we obtain a strengthened version of Lemma \ref{lem:311}, which will be useful in the proofs of Theorems \ref{main:3} and \ref{main:4}.
\begin{proposition}
\label{prop:44}
Let $\overline{a_1\ldots a_m}\in \mathcal F$, and let $s_1\ldots s_m$ be the Lyndon word for $a_1\ldots a_m$. Then any sequence in the set
\[
\mathcal{Z}_m:=\set{(x_i)\in\set{0,1}^\N: s_1\ldots s_m0^\f \lle \si^n((x_{i})) \lle (a_1\ldots a_m)^\f\quad\textrm{for all}\quad n\ge 0}
\]
ends with $(a_1\ldots a_m)^\f$ and $\# \mathcal Z_m < \infty$.
\end{proposition}

\begin{proof}
Let $\overline{a_1\ldots a_m}\in\mathcal F$.  Then $m\ge 2$,  $a_1=1$ and $a_m=0$.  We will prove this proposition by induction on $m$. If $m=2$, then $a_1a_2=10$. So $s_1s_2=01$ is the corresponding Lyndon word. The only two sequences $(x_i)$ satisfying
\[
010^\f\lle \si^n((x_i))\lle (10)^\f\quad\textrm{for all}\quad n\ge 0
\]
are $(01)^\f$ and $(10)^\f$. So the proposition holds for $m=2$.

Suppose the proposition is true for all $2\le m<N$ and all $\overline{a_1\ldots a_m}\in\mathcal F$. Let $(a_1\ldots a_N)^\infty\in \mathcal Q$ be such that $\overline{a_1\ldots a_N}\in \mathcal F$. The standard factorisation gives us a unique $m_1$ such that $a_1\ldots a_N= a_1\ldots a_{m_1} a_{m_1+1}\ldots a_N$ with $\overline{a_1\ldots a_{m_1}}$ and $\overline{a_{m_1+1}\ldots a_N}$ both Farey words. We consider two cases separately: (I) $m_1=1$ and (II) $2\leq m_1 <N$.\newline
Case (I): If $m_1=1$, Lemma \ref{lem:43} (ii) implies that the Lyndon word for $a_1\ldots a_N$ is
\[
a_Na_1\ldots a_{N-1}=0a_1\ldots a_{N-1}.
\]
Then by Lemma \ref{lem:311} it follows that any sequence $(x_i)$ beginning with digit $0$ and satisfying
\[
0a_1\ldots a_{N-1}0^\f\lle \si^n((x_i))\lle (a_1\ldots a_N)^\f\quad\textrm{for all}\quad n\ge 0
\]
equals $(0a_1\ldots a_{N-1})^\infty$. Observe that  for any sequence $(x_i)\in \mathcal Z_N$ there exists a $1 \le k\le N$, such that $x_k=0$, so $\sigma^{k-1}((x_i))=(0a_1\ldots a_{N-1})^\infty=0(a_1\ldots a_N)^\infty$. We find $\# \mathcal Z_N\le N$. This proves the proposition for $m_1=1$.\newline
Case (II): If $2\le m_1<N$, then Lemma \ref{lem:43} (i) and (ii) imply that the Lyndon word $s_1\ldots s_N$ for $a_1\ldots a_N$ is given by
\begin{equation}\label{eq:43}
s_1\ldots s_N=a_{N-m_1+1}\ldots a_Na_1\ldots a_{N-m_1}=0a_2\ldots a_{m_1}a_{m_1+1}\ldots a_{N-1}1.
\end{equation}
Moreover, by Lemma \ref{lem:43} (iii) we know that  $(a_1\ldots a_{m_1})^\f\in\mathcal Q$ and $\overline{a_1\ldots a_{m_1}} \in \mathcal F$. By the same lemma, the Lyndon word for $a_1\ldots a_{m_1}$ is given by
\begin{equation*}
 0a_2\ldots a_{m_1-1}1=s_1\ldots s_{m_1}^+.
\end{equation*}
Thus
\[
\mathcal{Z}_{m_1} =\set{(x_i)\in\set{0, 1}^\N: s_1\ldots s_{m_1}^+0^\f\lle \si^n((x_i))\lle (a_1\ldots a_{m_1})^\f\textrm{ for all }n\ge 0}.
\]
We claim that $\mathcal{Z}_{m_1}\cap\mathcal{Z}_N=\emptyset$. By induction any sequence in $\mathcal{Z}_{m_1}$ ends with $(a_1\ldots a_{m_1})^\f$. Now write $N=k m_1+j$ with $k\ge 1$ and $0<j\le m_1$. Then by Lemma \ref{lem:310} it follows that
 \begin{equation} \begin{split}\label{q:aNm1}
 a_1\ldots a_N&=(a_1\ldots a_{m_1})\cdots (a_{(k-1)m_1+1}\ldots  a_{k m_1})a_{N-j+1}\ldots a_N\\
 &\lle (a_1\ldots a_{m_1})^k a_{N-j+1}\ldots a_N \prec (a_1\ldots a_{m_1})^k a_1\ldots a_j.
 \end{split}\end{equation}
This implies that
 \begin{equation}\label{eq:*1}
 (a_1\ldots a_N)^\f\prec (a_1\ldots a_{m_1})^\f.
 \end{equation}
 Thus $(a_1\ldots a_{m_1})^\f\notin\mathcal {Z}_N$ and $\mathcal{Z}_{m_1}\cap\mathcal{Z}_N=\emptyset$. 

\vskip .2cm
Therefore, by (\ref{eq:43}) and (\ref{eq:*1}), it follows that for any sequence $(x_i)\in\mathcal Z_N=\mathcal Z_N\setminus\mathcal Z_{m_1}$ there exists a minimal $j\ge 0$ such that
 \[  x_{j+1}\ldots x_{j+m_1}=s_1\ldots s_{m_1}=a_{N-m_1+1}\ldots a_N. \]
By Lemma \ref{lem:311} we conclude for the tails that $x_{j+1}x_{j+2}\ldots =(s_1\ldots s_N)^\f$, which give the first part of the proposition. We now show that $\# \mathcal Z_N < \infty$. The minimality of $j$ implies that $x_1\ldots x_j\in\B_j(\mathcal Z_{m_1})$. Since $\# \mathcal Z_{m_1} < \infty$ and any sequence in $\mathcal Z_{m_1}$ ends with $(a_1 \ldots a_{m_1})^{\infty}$, there is a $k_1 \ge 0$, such that for each $(x_i) \in \mathcal Z_{m_1}$ the sequence $\sigma^{k_1} ((x_i))$ equals some shift of $(a_1 \ldots a_{m_1})^\infty$. Then for $j \ge k_1 + (k+2)m_1$ any $x_1 \ldots x_j \in \B_j(\mathcal Z_{m_1})$ contains the word $(a_1 \ldots a_{m_1})^{k+1}$. By \eqref{q:aNm1} this would contradict the fact that $(x_i) \in \mathcal Z_N$. Hence, $j < k_1 + (k+2)m_1$ and since for each $j$, the set $\B_j(\mathcal Z_{m_1})$ is finite, we get $\# \mathcal Z_N < \infty$.
\end{proof}

\begin{remark}{\rm Proposition~\ref{prop:44} shows that the size of $\mathcal Z_m$ is finite for each $m\ge 1$. Then it would be interesting to give a good upper bound of $\# \mathcal Z_m$. }
\end{remark}
 
\subsection{Farey intervals}
We now use the Farey words to identify the basic intervals $(\beta_L,\beta_R]$ that are not contained in any other basic interval.

\begin{definition}\label{def:45}
Let $\overline{a_1\ldots a_m}\in \mathcal F$ and let $\gamma_L$ and $\gamma_R$ be given by the quasi-greedy expansions $\al(\ga_L)=(a_1\ldots a_m)^\f$ and $\al(\ga_R)=a_1\ldots a_{m}^+(a_ma_{m-1}\ldots a_1)^\f$ respectively. Then the interval  $J_{a_1\ldots a_m}=(\ga_L, \ga_R]$ is called the \emph{Farey interval} generated by $a_1\ldots a_m$.
\end{definition}
   
Proposition \ref{prop:312} states that for any $\beta\in J_{a_1\ldots a_m}$ the set $E_\beta^+$ contains an isolated point. So the set of $\beta\in(1,2)$ for which $E_\beta^+$ has no isolated points is a subset of
\[ (1,2)\setminus\bigcup_{\overline{a_1\ldots a_m}\in\mathcal F} J_{a_1\ldots a_m}.\]
To prove Theorem~\ref{main:3} it is therefore enough to prove that this set has Hausdorff dimension zero. We do this by relating each Farey interval $J_{a_1 \ldots a_m}$ to another interval $I_{\overline{a_1\ldots a_m}}$ associated to the doubling map and using known results for the union of $\bigcup I_{\overline{a_1\ldots a_m}}$.

\vskip .2cm
Recall that the doubling map is given by $T_2(x)=2x \pmod 1$ and that $\pi_2: \{0,1\}^\mathbb N \to [0,1]$ is the projection map defined in $(\ref{eq:21})$. Set
\[ E_D:=\Big\{x\in \Big[0,\frac{1}{2}\Big) : T_2^n(x)\in \Big[x,x+\frac{1}{2}\Big] \text{ for all } n\ge 0\Big\}.\]
For each Farey word $w=w_1\ldots w_m\in \mathcal F$ we denote by  $I_w:=(q_L, q_R)$ the open interval associated to $w$, where
\[q_L=\pi_2((w_mw_{m-1}\ldots w_1)^\infty)-\frac{1}{2}\quad\textrm{and}\quad  q_R=\pi_2((w_1\ldots w_m)^\infty).\]
The interval $I_w=(q_L, q_R)$ is well-defined, since by (f1) it follows that 
\begin{align*}
q_L&=\pi_2(0w_{m-1}w_{m-2}\ldots w_1 (w_mw_{m-1}\ldots w_1)^\f)\\
&=\pi_2(w_1w_2\ldots w_{m-1}0(w_mw_{m-1}\ldots w_1)^\f) <\pi_2((w_1\ldots w_m)^\f)=q_R.
\end{align*}
In \cite{CIT13} we find the following result.
\begin{proposition}\cite[Proposition 2.14]{CIT13}\label{prop:48}
\begin{enumerate}
\item[{\rm(i)}]  Each $I_w$ is a connected component of $(0, \frac{1}{2})\setminus E_D$. Moreover,
\[
\Big(0,\frac{1}{2}\Big)\setminus E_D=\bigcup_{w\in\mathcal F}I_w.
\]

\item[{\rm(ii)}] 
 $\dim_H E_D=0$.
 \end{enumerate}
\end{proposition}

Recall that by Lemma~\ref{lem:21} the function $\alpha: \beta \mapsto \alpha(\beta)$ is a strictly increasing bijection from $(1,2]$ to $\mathcal Q$. Moreover, $\pi_2: \{0,1\}^\mathbb N \to (0,1]$ is a strictly increasing bijection if we remove from $
\{0,1\}^\mathbb N$ all sequences ending with $0^{\infty}$. Since such sequences do not occur as quasi-greedy expansions of $1$ and since the first digit $\al_1(\beta)$ equals 1 for any $\beta\in(1,2)$, the map
\[\phi: (1,2) \to  \Big(\frac12, 1 \Big), \,  \beta\mapsto \pi_2(\al(\beta))=\sum_{i=1}^\f\frac{\al_i(\beta)}{2^i}\]
is strictly increasing as well. The image $\phi ((1,2))$ is a proper subset of $(\frac{1}{2}, 1)$. 

\begin{lemma}\label{lem:49}
\[ \phi \Big((1,2)\setminus\bigcup_{\overline{a_1\ldots a_m}\in\mathcal F} J_{a_1\ldots a_m} \Big) \, \subseteq \, \Big(\frac12, 1\Big)\setminus\bigcup_{\overline{a_1\ldots a_m}\in\mathcal F}(1-I_{\overline{a_1\ldots a_m}}) = 1-E_D.\]
\end{lemma}

\begin{proof}
Let $a_1\ldots a_m$ be a word such that $\overline{a_1\ldots a_m}\in \mathcal F$. Note that 
\[
q_R=\pi_2((\overline{a_1\ldots a_m})^\infty)=\sum_{n\geq 1} \frac{1}{2^n} -\pi_2((a_1\ldots a_m)^\infty)=1-\phi(\ga_L).
\]
Moreover, by Lemma \ref{lem:43}(i) and (ii) it follows that
 \[
 \al(\ga_R)=a_1\ldots a_{m}^+(a_ma_{m-1}\ldots a_1)^\f=1a_{m-1}a_{m-2}\ldots a_1(a_ma_{m-1}\ldots a_1)^\f.
 \]
 Then
 \begin{align*}
 \phi(\ga_R)&=\pi_2(1a_{m-1}a_{m-2}\ldots a_1(a_ma_{m-1}\ldots a_1)^\f)\\
 &=\frac{1}{2}+\pi_2((a_ma_{m-1}\ldots a_1)^\f) =\frac{1}{2}+\big(1-\pi_2((\overline{a_ma_{m-1}\ldots a_1})^\f)\big)\\
 &=1-\big(\pi_2((\overline{a_ma_{m-1}\ldots a_1})^\f)-\frac{1}{2}\big)\;=1-q_L.
 \end{align*}
Since $\phi$ is strictly increasing and bijective from $(1,2)$ to $\phi((1,2))$, this implies that 
\[  \phi^{-1} \big((1-q_R,1-q_L) \big) = (\gamma_L, \gamma_R).\]
By Proposition~\ref{prop:48}(i) this gives the result.
\end{proof}

\begin{remark}
The proof of the previous lemma shows that each Farey interval is mapped precisely into one interval $I_w$ by the map $\phi$. By Proposition~\ref{prop:48}(i) we obtain that the Farey intervals are disjoint.
\end{remark}

Finally, to determine the Hausdorff dimension of $(1,2)\setminus\bigcup_{\overline{a_1\ldots a_m}\in\mathcal F} J_{a_1\ldots a_m}$, we prove that the inverse $\phi^{-1} : \pi_2 \circ \alpha ((1,2)) \to (1,2)$ is H\"older continuous and combine this with the following well known result: If $f: (X, \rho_1)\ra (Y, \rho_2)$ is a $c$-H\"{o}lder continuous map between two metric spaces $(X, \rho_1)$ and $(Y, \rho_2)$, then $\dim_H f(X)\le \frac{1}{c} \dim_H X$.
\begin{lemma}\label{lem:412}
For any integer $N\ge 2$ the function $\phi^{-1}$ is $c$-H\"older continuous with $c=\frac{\log(1+1/N)}{\log 4}$ on the  set  $\phi(\big[1+\frac{1}{N},2\big))$.
\end{lemma}

\begin{proof}
Fix $N\ge 2$ and let $\beta_1,\beta_2 \in [1+\frac{1}{N},2)$ with $\beta_1 < \beta_2$. Then $\al(\beta_1)\prec \al(\beta_2)$. Let $n$ be the first index such that $\alpha_n(\beta_1)<\alpha_n(\beta_2)$. Then 
\begin{equation}\label{eq:46}
\begin{split}
0< \beta_2-\beta_1 &= \beta_2 \sum_{j=1}^\infty \frac{\alpha_j(\beta_2)}{\beta_2^j} - \beta_1 \sum_{j=1}^\infty \frac{\alpha_j(\beta_1)}{\beta_1^j}  \\
& \le \ \sum_{j=n}^\infty \frac{\alpha_j(\beta_2)-\alpha_j(\beta_1)}{\beta_2^{j-1}} \leq  \sum_{j=n}^\infty \frac{1}{(1+\frac{1}{N})^{j-1}} \ =N \Big(1+\frac{1}{N}\Big)^{2-n}.
\end{split}
\end{equation}
 On the other hand,  we also have
\begin{equation}\label{eq:47}
\begin{split}
\pi_2\big(\alpha(\beta_2)\big)-\pi_2\big(\alpha(\beta_1)\big)&=  \sum_{j=1}^\infty \frac{\alpha_j(\beta_2)-\alpha_j(\beta_1)}{2^j} \ = \ \sum_{j=n}^\infty \frac{\alpha_j(\beta_2)-\alpha_j(\beta_1)}{2^j}\\
&\geq  \frac{1}{2^n}-\sum_{j=n+1}^\infty \frac{\alpha_j(\beta_1)}{2^j}  \geq \frac1{2^n(2^n-1)} \ > \ \frac{1}{4^n},\\
\end{split}
\end{equation}
where the second inequality follows, since
\[\al_{n+1}(\beta_1)\al_{n+2}(\beta_1)\ldots \lle \al_1(\beta_1)\al_2(\beta_1)\ldots \lle(1^{n-1}0)^\f.\]
Combining (\ref{eq:46}) and (\ref{eq:47}), we conclude that
\[
|\pi_2 \big( \alpha(\beta_2)\big)-\pi_2\big( \alpha(\beta_1)\big)|\geq  \frac{1}{4^n} = \Big(1+\frac{1}{N}\Big)^{-\frac{\log 4}{\log (1+\frac{1}{N})}n} \geq \Big( N \Big( 1+\frac1N \Big)^2 \Big)^{-\frac{\log 4}{\log (1+\frac1N )}}|\beta_2-\beta_1|^{\frac{\log 4 }{\log(1+\frac{1}{N})}}.
\qedhere \]
\end{proof}

\begin{proof}[Proof of Theorem~\ref{main:3}]
By Lemma~\ref{lem:49} the only thing left to show is that $\dim_H \phi^{-1}(1-E_D)=0$. This follows from Lemma~\ref{lem:412} and Proposition~\ref{prop:48} (ii) in the following way:
\begin{align*}
0 & \le  \dim_H \phi^{-1} (1-E_D)  =  \dim_H \left( \bigcup_{N \ge  2} \Big( \phi^{-1} (1-E_D) \cap \Big[1+\frac1N ,2 \Big) \Big) \right)\\
&= \sup_{N \ge 2} \dim_H \phi^{-1}\left(  (1-E_D) \cap \phi\left(\Big[1+\frac1N , 2\Big)\right) \right)\\
& \leq \sup_{N \ge 2} \, \frac{\log 4}{\log (1+1/N)} \dim_H\left((1-E_D) \cap \phi\left(\Big[1+\frac{1}{N}, 2\Big)\right)\right)\\
& \le \sup_{N \ge 2} \, \frac{\log 4}{\log (1+1/N)} \dim_H (1-E_D)  = \sup_{N \ge 2} \, \frac{\log 4}{\log (1+1/N)} \dim_H E_D =0. \qedhere
\end{align*}
\end{proof}

 \section{The critical points of the dimension function}\label{sec:5}
Since the map $\eta_\beta: t \mapsto \dim_H K_\beta(t)$ is a decreasing, continuous function with $\eta_\beta(0)=1$ and $\eta_\beta(\frac1\beta)=0$, there is a unique value $\tau_\beta$, such that $\dim_H K_\beta(t)>0$ if and only if $t < \tau_\beta$. Determining the value of $\tau_\beta$ would extend the results from \cite{Cla16} for holes of the form $(0,t)$. For $\beta=\gamma_L$ equal to the left endpoint of one of the Farey intervals, we show below that $\tau_\beta = 1-\frac1\beta$. This result is based on the following lemma. 
  \begin{lemma}\label{lem:53}
  Let $(a_1 \ldots a_m)^\f \in \mathcal Q$ with $\overline{a_1 \ldots a_m}\in\mathcal F$, and denote the Lyndon word for $a_1\ldots a_m$ by $s_1 \ldots s_m = a_{j+1}\ldots a_m a_1\ldots a_j$. For each $N \ge 1$, define the sequence $\t_N \in \{0,1\}^\infty$ by
  \begin{equation}\label{q:tN}
  \t_N:=(0a_2\ldots a_m(a_1\ldots a_m)^N a_1\ldots a_j)^\f.
  \end{equation}
Then for each $N \ge 1$, $\t_N\prec \t_{N+1}$. Furthermore, any sequence $\t$ that is a concatenation of blocks of the form 
\[0a_2\ldots a_m(a_1\ldots a_m)^k a_1\ldots a_j, \quad k \ge N,\]
satisfies $\t_N \lle \sigma^n(\t ) \prec (a_1 \ldots a_m)^\infty$ for all $n \ge 0$. In particular, we have for each $n \ge 0$ that
\[ \t_N \lle \sigma^n(\t_N ) \prec (a_1 \ldots a_m)^\infty .\]
\end{lemma}

\begin{proof} 
 By Lemma \ref{lem:43} it follows that 
  \begin{equation}\label{eq:51}
  s_1\ldots s_m=a_ma_{m-1}\ldots a_1=0a_2\ldots a_m^+=a_{j+1}\ldots a_m a_1\ldots a_j.
  \end{equation}
This implies that for all $N \ge 1$,
  \begin{align*}
  \t_N&=(0a_2\ldots a_m(a_1\ldots a_m)^N a_1\ldots a_j) (a_{j+1}\ldots a_m a_1\ldots a_j^-(a_1\ldots a_m)^Na_1\ldots a_j)^\f\\
  &\prec (0a_2\ldots a_m(a_1\ldots a_m)^{N+1}a_1\ldots a_j)^\f=\;\t_{N+1},
  \end{align*}
giving the first part of the statement. For the second statement, let $\t$ be a sequence consisting of a concatenation of blocks of the form $0a_2\ldots a_m(a_1\ldots a_m)^k a_1\ldots a_j$ with prefix $0a_2\ldots a_m(a_1\ldots a_m)^K a_1\ldots a_j$ for some $K \ge N$. We first show that $\sigma^n(\t) \prec (a_1 \ldots a_m)^{\infty}$ for all $n \ge 0$. For $n=0$ the statement is clear. By Lemma~\ref{lem:310} it follows that $a_{i+1} \ldots a_m \prec a_1 \ldots a_{m-i}$ for each $0 < i < m$. This implies that $\sigma^n(\t) \prec (a_1 \ldots a_m)^{\infty}$ for each $\ell m < n < (\ell+1)m$, $0 \le \ell \le K$. For all other values of $n < (K+1)m+j$ we obtain the result from \eqref{eq:51}, which implies that
\[ a_1\ldots a_j 0 a_2\ldots a_m= a_1 \ldots a_m a_1\ldots a_j^- \prec a_1\ldots a_m a_1\ldots a_j.\]
The same arguments then give the result for any $n \ge 0$. Hence, $\sigma^n(\t) \prec (a_1 \ldots a_m)^{\infty}$ for all $n \ge 0$. We now show that $\sigma^n(\t) \lge \t_N$ for each $n \ge 0$. Note that $\t$ has prefix
\[
s_1\ldots s_m^-(a_1\ldots a_m)^K a_1\ldots a_j.
\]
For $n =0$ the statement follows from \eqref{eq:51}. By (\ref{eq:51}), Lemmas \ref{lem:34} and  \ref{lem:310} it follows that   
\begin{equation*}
s_{i+1}\ldots s_m^-  \lge s_1\ldots s_{m-i}\quad\textrm{and}\quad
a_1\ldots a_i  \succ a_{m-i+1}\ldots a_m=s_{m-i+1}\ldots s_m^-
\end{equation*}
for all $0<i<m$, giving the statement for all $0 < n <m$. Since $s_1\ldots s_m$ is the Lyndon word associated to $a_1\ldots a_m$, we obtain
\begin{equation*}
a_{i+1}\ldots a_m a_1\ldots a_i\lge s_1\ldots s_m\succ s_1\ldots s_m^- \quad\textrm{for any}\quad 0\le i<m.
\end{equation*}
Since $a_1\ldots a_j s_1\ldots s_{m-j}=a_1\ldots a_m$, the conclusion that $\si^n(\t)\lge \t_N$ for all $n\ge 0$ follows.
 \end{proof}
 
\begin{proposition}\label{l:taugammaL}
Let $\overline{a_1 \ldots a_m}\in\mathcal F$ and let $\beta \in (1,2)$ be such that $\alpha(\beta) = (a_1 \ldots a_m)^{\infty}$. Then $1-\frac1\beta \in E_\beta^0$ and
\[ \tau_\beta = 1-\frac1\beta = \max \overline{E_\beta^+}.\]
\end{proposition} 
 
 \begin{proof}
Since $m$ is the minimal period of $\al(\beta)$, the greedy $\beta$-expansion of 1 is equal to $b(1, \beta) = a_1 \ldots a_m^+0^\f$. Lemma~\ref{lem:43} tells us that $a_1 \ldots a_m^+ = 1a_{m-1}\ldots a_1$, so
  \[
 \pi_\beta(a_ma_{m-1}\ldots a_10^\f)=\pi_\beta(1a_{m-1}\ldots a_10^\f)-\frac{1}{\beta}=\pi_\beta(a_1\ldots a_m^+0^\f)-\frac{1}{\beta}=1-\frac{1}{\beta}.\]
Recall that $a_m a_{m-1}\ldots a_1 =0a_2 \ldots a_m^+$. Then by Lemma~\ref{lem:310} it follows that for each $n\ge 0$, $\si^n(a_ma_{m-1}\ldots a_1 0^\f)\prec (a_1\ldots a_m)^\f=\al(\beta)$ and hence  $a_ma_{m-1}\ldots a_10^\f$ is the greedy $\beta$-expansion of $1-\frac1{\beta}$, i.e., $ b \big(1-\frac{1}{\beta}, \beta\big) =a_m a_{m-1}\ldots a_10^\f$. By Lemma~\ref{lem:34}, $b(1-\frac{1}{\beta}, \beta )\in\E_\beta^0$, so $1-\frac{1}{\beta}\in E_\beta^0$.

\vskip .2cm
The quasi-greedy $\beta$-expansion of $1-\frac{1}{\beta}$ is given by 
\[
\tilde b \Big(1-\frac{1}{\beta}, \beta \Big)=0a_2\ldots a_m(a_1\ldots a_m)^\f.
\]
Now consider the sequences $\t_N$ from Lemma~\ref{lem:53}. Since $\t_N \lle \sigma^n(\t_N ) \prec (a_1 \ldots a_m)^\infty = \alpha(\beta)$ for all $n \ge 0$, we have $\t_N \in \mathcal E_\beta^+$ for each $N \ge 1$. Moreover, if we set $t_N:=\pi_\beta(\t_N)$, then Lemma \ref{lem:22} gives that $t_N\nearrow 1-\frac{1}{\beta}$ as $N\ra\f$. So, $\max\overline{E_\beta^+}\ge 1-\frac{1}{\beta}$. Furthermore, the fact that any sequence of concatenations of blocks of the form $0a_2\ldots a_m(a_1\ldots a_m)^k a_1\ldots a_j$, $k\ge N$, belongs to $\mathcal K_\beta^+(t_N)$ implies that $h_{top}(\K_\beta^+(t_N))>0$ for all $N\ge 1$ and hence also $h_{top}(\K_\beta(t_N))>0$ for all $N\ge 1$. By the dimension formula \eqref{q:dimensionformula} we then get that $\tau_\beta \ge 1-\frac1{\beta}$.

\vskip .2cm
On the other hand, by Lemma~\ref{lem:43}(ii) and Proposition \ref{prop:44} we have
\begin{equation}\label{eq:58}
 \mathcal K_\beta^+ \Big(1-\frac{1}{\beta}\Big)=\set{(x_i): a_ma_{m-1}\ldots a_1 0^\f\lle\si^n((x_i))\prec (a_1\ldots a_m)^\f ~\forall n\ge 0} = \emptyset.
 \end{equation}
Since $E_\beta^+\cap[1-\frac{1}{\beta}, 1]\subseteq K_\beta^+(1-\frac{1}{\beta})$, this implies that $\max\overline{E_\beta^+} \le 1-\frac{1}{\beta}$. It also implies that $\dim_H K_\beta(1-\frac1\beta)=0$, which gives that $\tau_\beta \le 1-\frac1{\beta}$ and proves the result.
  \end{proof}

\begin{remark}\label{r:pos+}
{\rm Note that the previous lemma also implies that for any $t < 1-\frac1{\gamma_L}$ we have $h_{top}(\K_{\gamma_L}^+(t))>0$. We will use this later on.
}\end{remark}
 
Next we will give a lower and upper bound for $\tau_\beta$ on each Farey interval $(\ga_L, \ga_R]$. 
\begin{lemma}\label{lem:55}
Let $(a_1\ldots a_m)^\infty \in \mathcal Q$ with $\overline{a_1 \ldots a_m}\in\mathcal F$.  For each $\beta \in (\ga_L,\ga_R]$, set $ t^* = \pi_\beta(0a_2\ldots a_m (a_1\ldots a_m)^\f) $ and $t^\diamond = \pi_\beta(0a_2\ldots a_m^+ 0^\f)$. Then $t^* \in E_\beta^+$, $t^\diamond \in E_\beta^0$ and
\[ 1-\frac1{\beta} - \frac1{\beta^m} + \frac1{\beta(\beta^m -1)} \le t^* \le  \tau_\beta \le t^\diamond < 1-\frac{1}{\beta}.\]
\end{lemma}

\begin{proof}
Take $\beta\in (\ga_L, \ga_R]$. Then
\begin{equation*}
(a_1\ldots a_m)^\f\prec \al(\beta)\lle a_1\ldots a_m^+(a_m a_{m-1}\ldots a_1)^\f.
\end{equation*}
We first show that $ \tau_\beta \ge t^*$. By Lemmas~\ref{lem:43} and \ref{lem:310}, we have
\[
\sigma^n(0a_2 \ldots a_m (a_1 \ldots a_m)^\f)  \lle (a_1 \ldots a_m)^\infty \prec \alpha(\beta) \quad \forall \, n \ge 0.
\]
Hence, $b(t^*,\beta)=0a_2\ldots a_m (a_1 \ldots a_m)^\f$ and as in the proof of Lemma~\ref{lem:53} we have that $\sigma^n(b(t^*,\beta)) \lge b(t^*,\beta)$ for each $n \ge 0$. So $t^* \in E_\beta^+$.

\vskip .2cm
For each $t < t^*$ we have by Lemma \ref{lem:22} that $b(t, \beta)\prec 0a_2\ldots a_m (a_1 \ldots a_m)^\f$. This implies  that for $N$ large enough, $b(t, \beta ) \prec \t_N \prec (a_1 \ldots a_m)^\infty \prec \alpha(\beta)$. By Lemma~\ref{lem:53}, it follows that $t_N \in \mathcal K_\beta^+(t)$ and $h_{top} (\mathcal K_\beta(t)) \ge h_{top} (\mathcal K_\beta^+(t))>0$. Thus $\dim_H \mathcal K_\beta(t) > 0$ and $\tau_\beta \ge t^*$.

\vskip .2cm
On the other hand, for $t^\diamond$ we have that $0a_2\ldots a_m^+ 0^\f$ is admissible for any $\beta \in (\gamma_L, \gamma_R]$ and that $\sigma^n(0a_2\ldots a_m^+ 0^\f) \succ 0a_2\ldots a_m^+ 0^\f$ for all $0 < n < m$, so $t^\diamond \in E_\beta^0$. By Lemmas \ref{lem:43} and \ref{lem:39} we get
\begin{equation}\label{q:t*in} \begin{split}
\mathcal K_\beta^+(t^\diamond) & \subseteq\set{(x_i) : a_m a_{m-1} \ldots a_1 0^\f\lle \si^n((x_i)) \prec a_1\ldots a_{m}^+(a_ma_{m-1}\ldots a_1)^\f ~\forall n\ge 0 }\\
&=\set{(x_i): (a_m a_{m-1}\ldots a_1 )^\f\lle \si^n((x_i))\prec a_1\ldots a_{m}^+(a_ma_{m-1}\ldots a_1)^\f  ~\forall n\ge 0 }\\
&=\set{(x_i): (a_m a_{m-1}\ldots a_1 )^\f\lle \si^n((x_i))\lle (a_1\ldots a_m)^\f  ~\forall n\ge 0 }\\
&=\set{(x_i): a_m a_{m-1}\ldots a_1 0^\f\lle \si^n((x_i))\lle (a_1\ldots a_m)^\f  ~\forall n\ge 0 }.
\end{split}\end{equation}
By Proposition \ref{prop:44} it follows that $\# \K_\beta^+(t^\diamond) < \infty$, so that $\dim_H K_\beta(t^\diamond)=0$. This gives that $\tau_\beta \le t^\diamond$. Note that
\[ \pi_{\gamma_R} (a_1a_2\ldots a_m^+ (0a_2\ldots a_m)^\f)=1.\]
Then, we have for each $\beta \in (\gamma_L, \gamma_R]$, that
\[ t^* = \pi_\beta (0a_2\ldots a_m (a_1a_2\ldots a_m)^\f) >1 -\frac1{\beta}-\frac1{\beta^m} + \sum_{i = 1}^\infty \frac1{\beta^{im+1}} = 1-\frac1{\beta} -\frac1{\beta^m} +\frac1{\beta(\beta^m-1)}. \]
From Proposition~\ref{l:taugammaL} we know that $\pi_{\gamma_L} (0a_2 \ldots a_m^+ 0^\f) =1-\frac{1}{\ga_L}$. For $\beta > \gamma_L$ we have $a_1\ldots a_m^+ 0^\f\prec b(1, \beta)$, so that
\[ t^\diamond = \pi_\beta (0a_2\ldots a_m^+0^\f)  =\pi_\beta(a_1\ldots a_m^+0^\f)-\pi_\beta(10^\f)<1-\frac{1}{\beta}. \qedhere\]
\end{proof}

\noindent Below in Figure~\ref{fig:4} we see a plot of the lower and upper bounds for $\tau_\beta$ found in Lemma~\ref{lem:55}.
\begin{figure}[h!]\centering
\includegraphics[width=87mm]{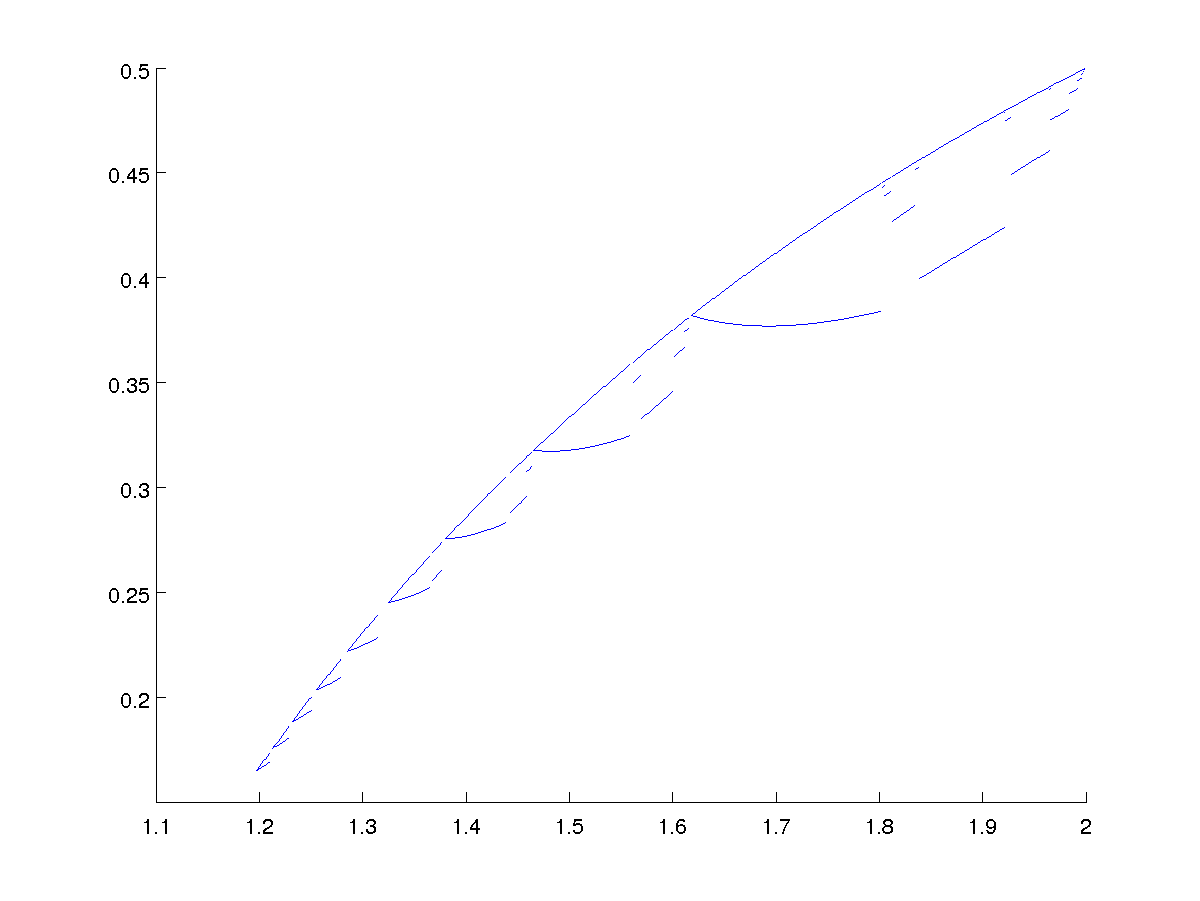}
 \caption{A plot of $1-\frac1{\beta}$ and $1-\frac1{\beta} -\frac1{\beta^m} +\frac1{\beta(\beta^m-1)}$ for basic intervals corresponding to Farey words of length $m$ with $m \le 10$.}\label{fig:4}
\end{figure}

The next lemma considers the critical point $\tau_\beta$ for the remaining values of $\beta$, i.e., those that are not in the closure of a Farey interval.
\begin{lemma}\label{l:criticalnotinfarey}
\label{lem:56}
Let $\beta\in(1,2)\setminus\bigcup[\ga_L, \ga_R]$ with the union taken over all Farey intervals.  Then $\max \overline{E_\beta^+}=\tau_\beta=1-\frac{1}{\beta}$.
\end{lemma}

\begin{proof}
Take $\beta\in(1,2)\setminus\bigcup[\ga_L, \ga_R]$. First we show that $\tau_\beta\ge 1-\frac{1}{\beta}$. Let $t<1-\frac{1}{\beta}$ with $b (t, \beta)=(b_i(t, \beta))$. Since $\dim_H \big( (1,2)\setminus\bigcup[\ga_L, \ga_R] \big)=0$, there exists a sequence of Farey intervals $([\ga_{L, k}, \ga_{R,k}])$, such that $\ga_{L, k}\nearrow \beta$ as $k\ra\f$. Thus, as $k\rightarrow \infty $ we have
  \begin{equation}\label{q:gammaLtobeta}
 \sum _{i=1}^\infty \frac{b_i (t, \beta)}{(\ga_{L, k})^i }\searrow \sum _{i=1}^\infty \frac{b_i (t, \beta)}{\beta ^i}=t
 \quad  \textrm{and}\quad 1-\frac{1}{\ga_{L, k}}\nearrow 1-\frac{1}{\beta }.
 \end{equation}
  For each $k$, we have a sequence $( \t_{k,N} ) \subseteq \mathcal E_{\gamma_{L,k}}^+$ as given in \eqref{q:tN}. Since $\gamma_{L,k} < \beta$, we obtain for each $N,n \ge 1$ that
\[ \t_{k,N} \lle \sigma^n (\t_{k,N}) \prec \alpha(\gamma_{L,k}) \prec \alpha(\beta).\]
Hence, $\t_{k,N} \in \mathcal E_\beta^+$ for all $k \ge 1$ and $N \ge 1$. This gives that $\max \overline{E_\beta^+} \ge 1-\frac1{\beta}$. Moreover, since $t<1-\frac{1}{\beta}$, we can find by \eqref{q:gammaLtobeta} a sufficiently large $M\in \mathbb N$, such that
 \begin{equation*}
  t< t_1:= \sum _{i=1}^\infty \frac{b_i (t, \beta)}{(\ga_{L, M})^i }<1-\frac{1}{\ga_{L, M}}< 1-\frac{1}{\beta }.
 \end{equation*}
Observe that  $b(t, \beta)=(b_i(t, \beta))$ is a $\ga_{L, M}$-expansion of $t_1$, which is lexicographically less than or equal to its greedy expansion $b(t_1,\ga_{L, M})$.   Then,
\begin{equation}\label{eq:512} 
\begin{split}
\mathcal K_\beta^+(t)&=\set{(x_i): b(t, \beta)\lle \si^n((x_i))\prec \al(\beta)~\forall n\ge 0}\\
&\supseteq\set{(x_i): b(t_1, \ga_{L, M})\lle \si^n((x_i))\prec \al(\ga_{L, M})~\forall n\ge 0}= \mathcal{K}_{\ga_{L,M}}^+(t_1).
\end{split}
\end{equation}
Since $\tau_{\ga_{L,M}}=1-\frac{1}{\ga_{L,M}} >t_1$, by Remark~\ref{r:pos+} we know that $h_{top}(\K_{\ga_{L,M}}^+(t_1))>0$ and together with \eqref{eq:512} we then find $h_{top}(\mathcal K_\beta^+(t))>0$, which in turn implies $\tau_\beta\ge t$. Since $t<1-\frac{1}{\beta}$ was taken arbitrarily, we conclude that $\tau_\beta \ge 1-\frac{1}{\beta}$.

\vskip .1cm 
To prove the other inequality we show that for any $t>1-\frac{1}{\beta}$ we have $\K_\beta^+(t)=\emptyset$. Take $t>1-\frac{1}{\beta}$. There is a sequence of Farey  intervals  $([\ga_{L,k}, \ga_{R,k}])$, such that $\ga_{L,k}\searrow \beta$ as $k\ra\f$. Thus, when $k\rightarrow \infty $ we have
\[
 \sum _{i=1}^\infty \frac{b_i (t, \beta)}{(\ga_{L,k})^i }\nearrow \sum _{i=1}^\infty \frac{b_i (t, \beta)}{\beta ^i}=t
 \quad \textrm{and}\quad 1-\frac{1}{\ga_{L,k}}\searrow 1-\frac{1}{\beta }.
\]
Since $t>1-\frac{1}{\beta}$, we can find a sufficiently large $N\in \mathbb N$ such that
 \begin{equation*}
 1-\frac{1}{\beta }< 1-\frac{1}{\ga_{L,N}}<t_2:=\sum_{i=1}^\f\frac{b_i (t, \beta)}{(\ga_{L,N})^i }<  t.
 \end{equation*}
Since $\ga_{L,N}>\beta$, $b(t, \beta)$ is the greedy $\ga_{L,N}$-expansion of $t_2$, i.e., $b(t, \beta)=b(t_2, \ga_{L,N})$.  Therefore, 
\[
\mathcal K_\beta^+(t) \subseteq\set{(x_i): b(t_2,\ga_{L,N})\lle \si^n((x_i))\prec \al(\ga_{L,N})~\forall n\ge 0}= \mathcal{K}^+_{\ga_{L,N}}(t_2) \subseteq \mathcal{K}_{\ga_{L,N}}^+(\tau_{\gamma_{L,N}}).
\]
From \eqref{eq:58} we conclude that $\K_\beta^+(t)=\emptyset$ and hence, $\max \overline{E^+_\beta}$, $\tau_\beta \le t$. Since $t>1-\frac{1}{\beta}$ was taken arbitrarily, we have $\max\overline{E_\beta^+}=\tau_\beta= 1-\frac{1}{\beta}$.
\end{proof}
 
\begin{proof}[Proof of Theorem~\ref{main:4}]
From Proposition~\ref{l:taugammaL}, Lemma~\ref{lem:55} and Lemma~\ref{l:criticalnotinfarey} we know that for all $\beta \in (1,2)$ we have $\tau_\beta \le 1-\frac1{\beta}$ with equality only if $\beta\in (1,2)\setminus\bigcup(\ga_L, \ga_R]$. We also know that for these points $\tau_\beta = \max \overline{E^+_\beta}$.

\vskip .2cm
By Proposition~\ref{prop:36} we know that any isolated point of $E^+_\beta$ has a periodic greedy $\beta$-expansion $b(t,\beta)$. From Proposition~\ref{prop:312} it follows that any $t \in (0,1)$ for which $b(t, \beta)=(s_1\ldots s_m)^\infty$ is periodic, is isolated in $E_\beta^+$ if and only if $\beta$ lies in the Farey interval generated by the largest cyclic permutation of $s_1 \ldots s_m$. So, if $\beta \not \in \bigcup(\ga_L, \ga_R]$, then $E_\beta^+$ cannot contain an isolated point and $\overline{E^+_\beta}$ is a Cantor set.
\end{proof}

\section{Final observations and remarks}

With the results from Theorems~\ref{main:2} and \ref{main:3} we have shown that the situation for $\beta \in (1,2)$ differs drastically from the situation for $\beta=2$, that was previously investigated in \cite{Urb86,Nil09,CT}. There are still several unanswered questions.

\vskip .2cm
Firstly, the structure of $E_\beta^0$ remains illusive to us. We know that $t \in E_\beta^0$ is isolated in $E_\beta$ if $\beta -1 \not \in K_\beta(t)$ and in Proposition~\ref{p:htopinE+} we proved that $h_{top}(\K_\beta(t)) = h_{top}(\K_\beta^+(t))$ for any $t \in E_\beta^+$. It would be interesting to know whether $t \in E_\beta^0$ is isolated in $E_\beta$ in case $\beta -1 \in K_\beta(t)$ and to consider $h_{top}(\K_\beta^0(t))$, also in case $t \not \in E_\beta^+$.

\vskip .2cm
In the previous section we have investigated the value of the critical point $\tau_\beta$ of the dimension function $\eta_\beta: t \mapsto \dim_H K_\beta(t)$. We could determine this value for any $\beta$ in the set $(1,2)\setminus\bigcup(\ga_L, \ga_R]$. If $\beta \in (\gamma_L, \gamma_R]$ for some Farey interval $(\gamma_L, \gamma_R]$, we only have a lower and upper bound for $\tau_\beta$. With a calculation very similar to the one in \eqref{q:t*in} one can show that for any $\beta \in (\gamma_L, \gamma_R]$ that satisfies
\[ \alpha (\beta) \prec a_1\ldots a_m^+ (0a_2 \ldots a_m)(a_1\ldots a_m)^{\infty},\]
we have $\tau_\beta = t^*$. However, for larger values of $\beta \in (\gamma_L, \gamma_R]$ the situation seems more intricate. It would be interesting to consider this question further by specifying $\tau_\beta$ more precisely also on $\bigcup (\gamma_L, \gamma_R]$ and by analysing the behaviour of the function $\tau: \beta \mapsto \tau_\beta$.

\section*{Acknowledgments}
The first author was partially supported by the NWO Veni-grant 639.031.140. The second author was supported by  NSFC No.~11401516. The fourth author was supported by NSFC No.~11671147, 11571144 and Science and Technology Commission of Shanghai Municipality (STCSM)  No.~13dz2260400.

 \bibliographystyle{alpha}
\bibliography{holes}
\end{document}